\documentclass[11pt]{amsart}
\usepackage{xcolor}
\usepackage{standard}
\usepackage{tikz}
\usetikzlibrary{matrix, calc, arrows}
\newcommand{\lag}{\mathcal{L}}
\newcommand{\torus}{\mathbb{T}}
\newcommand{\tN}{\widetilde{N}}

\newcommand{\plz}{+0}

\newcommand{\F}{\mathcal{F}}
\newcommand{\blue}[1]{\textcolor{black}{#1}}
\newcommand{\yellow}[1]{\textcolor{black}{#1}}
\newcommand{\teal}[1]{\textcolor{black}{#1}}
\newcommand{\red}[1]{\textcolor{black}{#1}}
\numberwithin{theorem}{section}
\newcommand{\tphi}{\widetilde{\phi}}
\title{Caustics of weakly Lagrangian distributions}
\author{Se\'an Gomes and Jared Wunsch}
\date{\today}
\begin{document}

\begin{abstract}
We study semiclassical sequences of distributions $u_h$ associated to a Lagrangian
submanifold of phase space $\lag \subset T^*X$.  If $u_h$ is a semiclassical
Lagrangian distribution, which concentrates at a maximal rate on
$\lag,$ then the asymptotics of $u_h$ are well-understood by work of
Arnol'd, provided
$\lag$ projects to $X$ with a stable simple Lagrangian singularity.  We
establish sup-norm estimates on $u_h$ under much more general
hypotheses on the rate at which it is concentrating on $\lag$ (again
assuming a stable simple projection).  These estimates
apply to sequences of eigenfunctions of integrable and KAM Hamiltonians.
  \end{abstract}
\maketitle
\section{Introduction}
Let $X$ be a smooth $n$-dimensional manifold.  Let
$p(x,\xi) \in \CI(T^*X;\RR)$ be a Hamiltonian function, and
$P_h \in \Psi_h(X)$ a self-adjoint pseudodifferential operator with
principal symbol $p.$ If the Hamilton flow associated to $p$ is
integrable, the phase space $T^*X$ is foliated by invariant
Arnol'd--Liouville Lagrangian tori on which the flow is quasi-periodic
\cite{Arnold}; if $p$ is a perturbation of an integrable Hamiltonian,
the KAM theorem \cite{Ko:54}, \cite{Ar:63}, \cite{Mo:62}
ensures that certain invariant tori on which the
frequencies of motion satisfy a Diophantine condition still survive
the perturbation.

Now let $u_h$ be a sequence of eigenfunctions of $P_h,$ i.e., $P_h u_h=E_h u_h$
with $h\downarrow 0,$ and where $E_h=E+O(h).$ We recall that the
\emph{semiclassical wavefront set}  $\WF_h u_h$ is a measure of where, in phase
space, a sequence of eigenfunctions may concentrate as $h \downarrow 0$, and that it is
known to lie in the characteristic set $\{p=E\},$ and to be invariant
under the Hamilton flow of $p.$ $\WF_h u_h$ may thus concentrate on a single
Arnol'd-Liouville torus in integrable or near-integrable systems, and
in the case of the Diophantine tori in the latter setting, may not
concentrate on any proper subset (as it is closed and invariant under
an irrational flow).  \blue{Sequences of eigenfunctions of this type
  are thus the quantum analogue of classical states that have
  well-defined values of the commuting variables, in the integrable
  case, or that remain in quasi-periodic motion in the KAM setting.}
Some research has been devoted to
understanding the properties of these sequences of eigenfunctions concentrating on
Lagrangian tori; for instance Galkowski--Toth
\cite{GaTo:18} studied sup-norm estimates in the case in which the
system is \emph{quantum} completely integrable, with the
eigenfunctions being joint eigenfunctions of a family of commuting
operators whose symbols cut out the invariant torus. Very little is
known in the KAM case, however.

In this paper, we study the most general setting in which a family of
eigenfunctions $u_h$ may concentrate along a Lagrangian submanifold
$\lag$ of $T^*X.$ In particular, \emph{we do not assume that $u_h$ is
  a Lagrangian distribution,} i.e.\ it does not necessarily enjoy
semiclassical Lagrangian regularity; this notion (defined below) would presuppose
that the rate of concentration of $u_h$ along $\lag$ occurs at maximal
possible rate.  By contrast, we will only assume that there is
\emph{some} quantitative rate of concentration on $\lag,$ and our
results reflect this rate explicitly.  \blue{The sup norm estimates also
depend (as is well-known in the case of Lagrangian distributions) on the \emph{singularities} of the projection to the
base of the Lagrangian in question. The critical values of
the projection map $\pi:\lag \to X$ 
are referred to as a \emph{caustic}, and the concentration of mass of
$u_h$ near such points is a familiar phenomenon from everyday life, for
instance in the brighter image of a light source on the surface of one's tea
at points where rays are focused by the side of cup.  The study of such phenomena has
a long history---see, e.g., \cite[f.87]{Le}.}
While in general the critical values of $\pi$
may be quite wild, we confine our attention here to the
finite list of \emph{stable simple singularities} developed by Arnol'd
\cite[Corollary 11.5]{Ar:72}; in dimension not exceeding $5$, every
Lagrangian projection can be perturbed to have a singularity in this
list \cite[Corollary 11.7]{Ar:72}.  In the case of actual Lagrangian
distributions, our results reduce to the classical descriptions of the
asymptotics of caustics in \cite{Ar:72a},
\cite{Duistermaat:Oscillatory}, \cite{Guillemin-Sternberg1}.  By
contrast, our results are nontrivial even in the case where $\lag$
projects diffeomorphically onto the base (see \S\ref{sec:tori} below),
as the rate of concentration on the torus affects the rate of growth
strongly in every case.

We measure the rate of concentration of $u_h$ along $h$ by
an \emph{iterated regularity} definition.   Let us suppose that we
normalize to $\norm{u_h}_{L^2}=1.$   If the Lagrangian were
simply $\lag \equiv \{ x=0 \}\subset T^*\RR^n,$ the rate at which a family of
distributions concentrates on $\lag$ could be given by asking how much smaller
$x^\alpha u_h$ is than $u_h$ as $h \downarrow 0;$ we might, for
instance, ask that
$$
\smallnorm{x^\alpha u_h}_{L^2}=O(h^{(1-\delta) \smallabs{\alpha}}),
$$
for some $\delta \in [0,1].$  This is a special case of the following general
definition.  \blue{In what follows, $\Psi_h^{-\infty}(X)$ denotes the
algebra of semiclassical pseudodifferential operators on $X$ with
rapidly-decreasing symbols, and $\sigma_h\colon \Psi_h^{-\infty}(X)
\to \CI(T^*X)$ denotes the principal symbol map \cite[Chapter 14]{Zw:12}.}

\begin{definition}
  Let $\lag \subset T^*X$ be a compact Lagrangian submanifold and let
  $\delta \in [0,1].$ We say that $u_h$ is a  $\delta$-Lagrangian distribution with respect to $\lag,$ if for all $N$ and all
  $A_1, \dots A_N \in \Psi_h^{-\infty}(X)$ such that $\sigma_h(A_j)=0$
  on $\lag,$ $u_h$ enjoys the iterated regularity property
  $$
\norm{A_1 \dots A_N u_h}_{L^2(X)} \leq C_N h^{N(1-\delta)},\quad h \in (0, 1).
$$
\end{definition}
When $\delta=0$ this is the usual definition of semiclassical
Lagrangian regularity---cf.\ \cite{Alexandrova:Semiclassical}.  When
$\delta=1$ the definition is satisfied for any $u_h \in L^2(X).$ 
For
intermediate values of $\delta$ we thus have a notion of partial
Lagrangian regularity, encoding a concentration of the states in
question on a Lagrangian submanifold at a variable rate.
(We do not consider $\delta>1,$ as this would not be achievable with $u_h$ compactly
microsupported, by the uncertaintly principle.)

Our main results are local sup-norm estimates for a semiclassical family of
distributions $u_h$ that are $\delta$-Lagrangian with respect to
$\lag,$ where $\lag$ has a singular projection given by one of the
stable simple singularities listed in Table~\ref{table:classification}
below.  There are two versions of these estimates: in the first, we
make no further assumptions, but in the second, stronger, estimate, we additionally assume that $u_h$
satisfies an approximate eigenfunction equation (where we have now absorbed the
eigenparameter into the operator)
$$
P_h u_h=O_{L^2}(h)
$$
where $\sigma(P_h)=0$ on $\lag.$ Our estimates all involve a constraint on
$\delta:$ it cannot exceed a threshold $\delta_0$ that depends on the
form of the caustic (but is equal to $1$ in the nonsingular case).
Beyond this threshold, the phenomenology seems intriguingly different,
and for the special case of the fold singularity, we also give
estimates for $\delta>\delta_0,$ and see that there is indeed a change
of qualitative behavior of extremizers (\S\ref{sec:beyond}).

In the next section, we describe our results in the special case of
the rectangular flat torus.  In this setting, they are far from
sharp, with improvements available using number-theoretic
tools.
We then recall the general geometric setting of stable simple Lagrangian
singularities, and proceed to the proofs of the main theorems.  The
main ingredients here are, first, a recapitulation of the
H\"ormander--Melrose theory of Lagrangian distributions in the setting
considered here, with limited regularity. This allows us to write
a $\delta$-Lagrangian distribution $u_h$ as an oscillatory integral in which
the amplitude function is not uniformly smooth as $h\downarrow 0$ but
rather lies in an $h$-dependent symbol class satisfying
$$
h^{-\delta \smallabs{\alpha}} \pa^\alpha a \in h^{-\gamma} L^\infty
$$
for some $\gamma.$
We then estimate the size of the function on the caustic by estimating
the resulting oscillatory integral.  This integral estimate is well-known
when $\delta=0$ (i.e., the standard Lagrangian case)---see  \cite{Ar:72a},
\cite{Duistermaat:Oscillatory}, \cite{Guillemin-Sternberg1}.  In the
case at hand, however, the usual proof of this classical result fails
to yield a sharp result: 
it employs the
Malgrange Preparation Theorem in an essential way, and this entails a
hard-to-quantify number of derivatives falling on the amplitude,
incurring $h^{-\delta}$ penalties each time.  We thus employ a
different, cruder method that so far as we know is novel, where we
split the integral into pieces to
estimate sup-norms rather than obtaining the precise asymptotics \blue{along
the caustic} that are
part of the classical theory.

Our main result is as follows.
\begin{theorem}\label{theorem:main}
  Let $u_h$ be a $\delta$-Lagrangian distribution with respect to a
  Lagrangian $\lag,$ microsupported in a set where the projection
  of $\lag$ has a singularity that is Lagrange-equivalent to one of
  the \yellow{stable simple} singularities listed in
  Table~\ref{table:intro-orders}.  Assume that $\delta<\delta_0$ for
  the corresponding threshold $\delta_0$ listed in the table.  Then
  \blue{there exists $C$ such that for all $h \in (0,1),$}
  $$
\frac{\smallnorm{u_h}_{L^\infty}}{\smallnorm{u_h}_{L^2}}\leq C h^{-\kappa-n\delta/2}
  $$
  where $\kappa$ is the order listed in
  Table~\ref{table:intro-orders}.

  If it is further the case that
  $$
P u=O(h)
$$
where $P$ is an operator of real principal type whose principal symbol
vanishes on $\lag,$ then \blue{for all $\ep>0$ there exists $C_\ep$ such that for all $h \in (0,1),$}
  $$
\frac{\smallnorm{u_h}_{L^\infty}}{\smallnorm{u_h}_{L^2}}\leq C_\ep h^{-\kappa-(n-1)\delta/2-\ep}.
  $$
\end{theorem}

\def\arraystretch{1.5}
\begin{table}\label{table:intro-orders}
\begin{equation*}
\begin{array}{||l|l|l||}\hline\hline
  \text{Type}\ &  \text{Order}\ \kappa & \text{Threshold}\ \delta_0\\ \hline
A_{m+1} & \frac 12-\frac{1}{m+2}& \frac{1}{m+2},\ (m>0);\  1,\ (m=0)\\ \hline
  D_{m+1}\ (m \text{ even}), D^-_{m+1}\ (m\text{ odd}) & \frac 12 -\frac{1}{2m}& \frac 1{m+1}\\ \hline
    D^+_{m+1}\ (m\text{ odd})& \frac 12 -\frac{1}{2m}& \frac 1m\\ \hline
  E_6 &  \frac 5{12}& \frac 16\\ 
  \hline E_7 & \frac 49 &\frac 17\\
  \hline E_8 & \frac{7}{15} & \frac 18\\
\hline\hline
\end{array}
\end{equation*}
\caption{Orders of caustics and thresholds of Lagrangian regularity.}
\end{table}

\thanks{The authors are grateful to Steve Zelditch for helpful
  discussions and to Ilya Khayutin for explaining
  the number-theoretic literature on lattice point counting in
  shrinking spherical caps (Section~\ref{sec:tori}).
  \yellow{St\'ephane Nonnenmacher as well as two anonymous
  referees made many helpful suggestions on the exposition; one of
  the latter
pointed out an error in the inductive step proving the main theorem.
  JW gratefully
  acknowledges partial support from Simons Foundation grant 631302 and
  from NSF grant DMS--1600023.
}

\section{Flat tori}\label{sec:tori}
As an illustration of the effects of weak Lagrangian regularity on
sup-norm estimates in a geometrically simple setting, we directly
prove
our main results in the
special case of square flat tori: $X=\RR^n/2\pi \ZZ^n.$
For each $\alpha \in (\RR^n)^*,$ let $e_\alpha(x)=e^{-i\alpha x}$ denote the corresponding complex exponential.

Fix a frequency vector $\omega \in (\RR^n)^*.$  \blue{Employing canonical
coordinates $(x,\xi)$ on $T^*X,$} we will consider the
Lagrangian
$$
\lag=\{\xi=\omega\} \subset T^* X.
$$
A normalized $\delta$-Lagrangian sequence is thus a sequence of functions $u_j$ on $\torus^n$ such that
$$
\norm{u_j}_{L^2}=1
$$
and such that for appropriately chosen $h\equiv h_j \downarrow 0$ and
any $N$ and choice of indices $k_1,\dots k_N \in \{1,\dots, n\},$
\begin{equation}\label{iterativetorus}
\big(h^{-1+\delta} (h D_{k_1} -\omega_{k_1})\big)\dots \big(h^{-1+\delta} (h D_{k_N} -\omega_{k_N}) \big) u_{\red{j}}=O_{L^2}(1)\text{ as } j \to \infty.
\end{equation}
\blue{We return to the notation $u_h$ for the sequence of functions, bearing in
  mind that $h=h_j \downarrow 0$ through a discrete sequence of values.}
(Note that the general definition of Lagrangian regularity would allow \emph{any} operators
characteristic on $\lag,$ rather than the specific operators
$hD_j-\alpha_j$ used here; however by elliptic regularity, it suffices
to consider just this set of test operators whose symbols are a set of
defining functions for $\lag.$)
Note that one immediate consequence of the assumption
\eqref{iterativetorus} is a crude $L^\infty$ estimate based on Sobolev
embedding: this estimate yields $D^\alpha
u_h=O_{L^2}(h^{-\smallabs{\alpha}}),$ hence certainly
\blue{\begin{equation}\label{torussobolev}
\sup \smallabs{u_h} =O(h^{-n/2+\ep})\smallnorm{u_h}_{L^2}
\end{equation}
for all $\ep>0.$}

We now write $u_h$ as the Fourier series
$$
\sum_{\alpha\in \ZZ^n} a_\alpha(h) e_\alpha(x).
$$
Fixing any $\delta'>\delta,$ we split
$$
u_h=v_h+w_h
$$
where
\begin{align*}
v_h&=\sum_{\smallabs{\alpha-h^{-1}\omega}<h^{-\delta'}} a_\alpha(h)
e_\alpha(x),\\
w_h&=\sum_{\smallabs{\alpha-h^{-1}\omega}\geq h^{-\delta'}} a_\alpha(h) e_\alpha(x).
\end{align*}
Since they are orthogonal, the estimate \eqref{iterativetorus} applies
to both $v_h$ and $w_h$ separately.
Taking $k_j=k$ all the same, this yields for the Fourier series of $w_h$ the estimate (for
each $k$)
$$
\sum_{\smallabs{\alpha-h^{-1}\omega}\geq h^{-\delta'}} \big[h^{\delta} (\alpha_k-\omega_k/h)\big]^N\smallabs{a_\alpha}^2=O(1);
$$
adding up the estimates for $k=1,\dots, n$ and using the comparability
of $\sum_1^n \smallabs{x_j}^N$ and $\smallabs{x}^N$ yields
$$
\sum_{\smallabs{\alpha-h^{-1}\omega}\geq h^{-\delta'}} \big[h^{\delta} \abs{\alpha-\omega/h}\big]^N\smallabs{a_\alpha}^2=O(1),
$$
i.e.,
$$
\sum_{\smallabs{\alpha-h^{-1}\omega}\geq h^{-\delta'}} h^{N(\delta-\delta')} \smallabs{a_\alpha}^2=O(1),
$$
hence
$$
\norm{w_h}_{L^2}=O(h^\infty).
$$
By \eqref{torussobolev}, then
$$
\norm{w_h}_{L^\infty}=O(h^\infty),
$$
and we need only consider $v_h$ in our estimates henceforth.

To estimate $v_h,$ we let
$$
N_{\red{\mu}}(h) = \#\big\{ \alpha \in \ZZ^n\colon\smallabs{\alpha-h^{-1} \omega}< h^{-\red{\mu}}\big\}
$$
\red{for $\mu\in(0,1]$.} \red{From the leading term in the Gauss circle problem, we have } $N_{\mu}(h)\sim C h^{-n \mu}$ \red{for a constant $C>0$ that depends only on $n$.}  Thus, since $u_h$ is $L^2$-normalized, we easily see by Cauchy--Schwarz that 
$$
\norm{v_h}_{L^\infty} \leq \sqrt{N_{\delta'}(h)}=O(h^{-n\delta'/2}).
$$
We have thus obtained
$$
\norm{u_h}_{L^\infty} \leq \sqrt{N_{\delta'}(h)}=O(h^{-n\delta/2-\red{\ep}})
$$
\red{for any $\ep>0$, as $\delta'>\delta$ can be chosen arbitrarily}.
This bound is achieved (up to an epsilon power) by taking all
$a_\alpha=N_{\red{\delta'}}(h)^{-1/2}$ for $\alpha$ such that
$\abs{\alpha-\omega/h}\leq Ch^{-\delta'},$ and zero otherwise.

This is, up to a loss of $h^{-\epsilon},$ precisely the special case of Theorem~\ref{theorem:main} for projectable
Lagrangians (the case $A_1$).
 When $\delta=1$ we essentially get the counting function for
 eigenfunctions in a large ball, but when $\delta=0$ we get $O(1),$
 the estimate for actual Lagrangian distributions associated to a projectable Lagrangian.

Note that we could recover the $\epsilon$ lost
here relative to the sharp statement of Theorem~\ref{theorem:main} by using Cauchy--Schwarz, somewhat as in
Lemma~\ref{lemma:semisobolev} below.  We have preferred to give a
treatment that emphasizes the role of simply counting lattice points
in domains in $\RR^n$, however; in
particular, this point of view makes the improvement in the result
very clear when we assume that the $u_{h_j}$ are Laplace eigenfunctions,
i.e.,
$$
(h_j^2\Lap-1)u_{h_j}=0.
$$  The point is that this gives us more precise localization in one
direction (conormal to the characteristic set).  In that case, $v_h$
now consists only of a sums as above with the further constraint $\smallabs{\alpha}=h^{-1},$
hence the $L^\infty$ estimate is replaced by $\sqrt{\tN_{\delta'}(h)}$
where $\delta'>\delta$ and
\begin{equation}\label{tN}
  \tN_{\red{\mu}}(h) = \#\big\{ \alpha \in \ZZ^n\colon \smallabs{\alpha}=h^{-1},\ \smallabs{\alpha-h^{-1} \omega} \leq C h^{-\red{\mu}}\big\}
\end{equation}
\red{for $\mu\in(0,1]$.} (Now of course we take $\omega$ only with $\smallabs{\omega}=1.$)
This quantity is a little subtler to estimate than $N_{\red{\mu}}(h).$

To obtain an improved upper bound on $\tN_{\red{\mu}}(h),$ we note that just as
with the usual Gauss method for the circle problem, we may bound it by
the sum of volumes of unit boxes centered at all lattice points in the
set on the right side of \eqref{tN}, and that this is in turn bounded
by the volume of the set
$$
\{ \alpha \in \RR^n\colon \lvert \smallabs{\alpha}-h^{-1}\rvert<C,\ \smallabs{\alpha-h^{-1} \omega} \leq C h^{-\red{\mu}}\big\}.
$$
(Indeed, this estimate applies even if $u_h$ is an $O(h)$ quasimode of
$h^2\Lap-1$.)
The result is comparable to the volume of the subset of the sphere of
radius $h^{-1}$ on which $\smallabs{\alpha-h^{-1} \omega} \leq C
h^{-\red{\mu}},$ i.e.\ we get
\begin{equation}\label{tNbound}
\tN_\mu(h) =O(h^{-(n-1)\mu}).
\end{equation}
Thus, using this estimate for $\tN_{\delta'}$ on the function $v_h$ in
our splitting, yields a sup-norm estimate for eigenfunctions (which would also apply
for $O(h)$ quasimodes) as follows:
\begin{equation}\label{torus:eigenfunction}
\norm{u_h}_{L^\infty} \leq \sqrt{\tN_{\red{\delta'}}(h)}=O(h^{-(n-1)\delta/2-\ep})
\end{equation}
\red{for any $\ep>0$, as $\delta'>\delta$ can be chosen arbitrarily.}
Again this recovers a special case of Theorem~\ref{theorem:main}.  But
this result is not, in this special case, optimal. We motivate the
optimal result by a crude lower bound.
\begin{lemma}\label{lemma:toruslattice}
  For any $\delta \in (0,1]$ and in any dimension $n \geq 1,$ there exists a sequence of $h \downarrow 0$ such that
  $$
\tN_\delta(h) \geq C h^{1-(n-1)\delta}.
$$
  \end{lemma}
\blue{Setting
    $$
f_h = \sum_{\substack{\smallabs{\alpha}=h^{-1}\\ \smallabs{\alpha-h^{-1} \omega} \leq C
    h^{-\delta}}} e_\alpha
$$
yields
$$
\norm{f_h}_{L^\infty} = f_h(0) = \tN_\delta(h)
$$
and, by orthogonality,
$$
\norm{f_h}_{L^2} = \sqrt{\tN_\delta(h)}.
$$
Thus, setting $u_h=f_h/\smallnorm{f_h}_{L^2},$ Lemma~\ref{lemma:toruslattice} shows that for an
$L^2$-normalized 
$\delta$-Lagrangian sequence of Laplace eigenfunctions on the torus we can achieve
\begin{equation}\label{LB}
\norm{u_h}_{L^\infty}\geq C h^{1/2-(n-1)\delta/2}.
\end{equation}}
  \begin{proof}[Proof of lemma]
    For $j \in \NN,$ let $$M(j) =\#\big\{ \alpha \in \ZZ^n\colon \smallabs{\alpha}^2=j,\ \smallabs{\alpha-j^{1/2} \omega} \leq j^{\delta/2}\}.$$  Thus,
    $$
\tN(j^{-1/2}) = M(j).
$$
Now
\begin{equation}\label{latticesum}
\sum_{J \leq j \leq 2J} M(j)=\# \ZZ^n \cap \Omega_J
\end{equation}
where
$$
\Omega_J \equiv  \big \{ r \theta \in \RR^n\colon r \in [\sqrt{J},
\sqrt{2J}],\ \smallabs{\theta-r\omega}<r^\delta \big \}.
$$
The quantity \eqref{latticesum} is comparable to the volume of the
solid in question (again by counting enclosed unit cubes), hence
$$
\sum_{J \leq j \leq 2J} M(j) \geq  C\int_{\sqrt{J}}^{\sqrt{2J}} (r^\delta)^{n-1} \, dr \sim C J^{1/2+(n-1)\delta/2}.
$$
On the other hand there are $J$ terms in the sum, so one of them must be at least
$$
 J^{-1/2+(n-1)\delta/2}.
 $$
Using this procedure to pick a sequence of $h=j^{-1/2}$ in the dyadic intervals $(J,2J)=(2^k, 2^{k+1})$ gives the desired sequence.
  \end{proof}

In dimension $n \geq 5,$ if for $m \in \NN$ we let $r_n(m)$ denote the number
of integer lattice points on the sphere of radius $m^{1/2},$  it is known
that there exist positive constants $c_n, C_n$ such that
$$
c_n m^{n/2-1} \leq r_n(m) \leq C_n m^{n/2-1}
$$
Thus, the number of lattice points on the sphere of radius $h^{-1}$ is
comparable to $h^{-n+2}$ for $n \geq 5.$
If we then multiply by the fraction of the volume of the sphere that
is occupied by the cap of size $h^{-\delta}$ we obtain a heuristic estimate
exactly of order $h^{1-(n-1)\delta}.$  This is indeed also known
to be essentially an upper bound, for sufficiently large $\delta$:
Bourgain--Rudnick \cite[Proposition 1.4]{BoRu:09}
show that for $n \geq 5,$
for $\delta \in [1/2,1],$
for all
$\ep>0$ there exists
$C=C_\ep$ such that for all $h$
$$
\tN_\delta(h) \leq C h^{1-(n-1)\delta-\ep}.
$$
(Similar results for the special cases $n=3,4$ are also obtained in \cite{BoRu:09}.)
Optimal lower bounds on $\tN_\delta(h)$ of the form \red{of the first equation in Lemma \ref{lemma:toruslattice}} (uniform in
radius, rather than along a subsequence as deduced above) have
recently been obtained by Sardari \cite[Corollary 1.9]{Sa:19};
see also the celebrated work of Duke \cite{Du:88} and Iwaniec
\cite{Iw:87} in the special case of dimension $3$.

  

\section{\yellow{Stable simple} singularities of Lagrangian projections}\label{section:stable}

\yellow{We now return to the general geometric setting of a
  non-projectable Lagrangian (i.e., the projection map is not assumed
  to be a diffeomorphism), and recall the normal forms of \emph{stable
    simple} singularities of Lagrangian projections as developed by
  Arnol'd \cite[Corollary 11.8]{Ar:72}, \cite{Ar:72a}.
  We will in fact use the alternative parametrizations of the
  Lagrangians given by Duistermaat \cite[Theorems~3.1.1 and
  3.2.1]{Duistermaat:Oscillatory}.  We recall first the notion of
  local \emph{Lagrange-equivalence}: two Lagrangians in $T^*X$
  are locally equivalent if they can be mapped one to another by
  a fiber-preserving local symplectomorphism of $T^*X.$
  \emph{Stability} of a
  Lagrangian projection means that nearby (in the
  $\mathcal{C}^\infty$ topology) Lagrangians are locally Lagrange-equivalent
  to the original.  The \emph{simple} singularities are those
  that under
  perturbation can be locally equivalent to only a finite list of
  singularities at nearby points \cite[Definition 11.1]{Ar:72}.
  Stability does not imply simplicity nor conversely in general, but
  stability does imply simplicity in dimension up to $5.$
  Thus the classification is in fact an exhaustive list of the stable
  singularities in these dimensions; moreover every Lagrangian in dimension up
  to $5$ can be locally perturbed to be equivalent to one in this
  list (stable Lagrangians are dense).  We refer the reader to \cite{Duistermaat:Oscillatory} and
  to \cite{Ar:72}, \cite{Ar:72a} for further details on the notions of stability and
  simplicity, and the classification.}

We recall that every Lagrangian manifold $\lag$ of $T^*\RR^n$ may locally be
parametrized in the following form:
$$
\lag=\big\{(x, \phi'_x(x,\theta)) \colon \phi'_\theta(x,\theta)=0\big\}.
$$
\yellow{Two phase functions $\phi$ and $\tphi$ are
easily seen to parametrize Lagrange-equivalent Lagrangians if
\begin{equation}\label{equiv}
  \tphi(x,\theta)=\phi(x', \theta')+\psi(x')
\end{equation}
  for some fiber-preserving local diffeomorphism $$(x,\theta) \mapsto
  (x'(x), \theta'(x,\theta)),$$ and $\psi
  \in\mathcal{C}^\infty.$  In \cite{Duistermaat:Oscillatory} this is
  referred to as \emph{equivalence of unfoldings} of the Lagrangian
  singularities, and it is as a classification of unfoldings up to the
  equivalence \eqref{equiv} that the classification is
  phrased in that work and in this form that we will employ it: every
phase function parametrizing a stable simple singularity is locally
equivalent to one in the Table~\ref{table:classification} (whose entries we
explain below) in the sense
\eqref{equiv}.}

Duistermaat \cite{Duistermaat:Oscillatory} parametrizes the
stable simple singularities in $\RR^n$ with phase functions
$$
\phi(x,\theta) = \sum_{j=1}^n x_j  f_j(\theta) + f(\theta)
$$
where $f_j,$ $f$ are given by Table~\ref{table:classification} (taken from \cite[Theorem
3.1.1 and Theorem 3.2.1]{Duistermaat:Oscillatory}); here $n$ is the
dimension, and $k$ is the number of phase variables $\theta$
(whose least possible value for each singularity is listed in the table); the $f_j$'s beyond
those enumerated ($f_1,\dots, f_m$ for the $A_{m+1}$ and $D^\pm_{m+1}$) are
taken to equal $0;$ the variables $\theta'$ are the remaining
$\theta\in \RR^k$
variables beyond those appearing explicitly ($\theta_2,\dots, \theta_k$ for $A_{m+1};$ $\theta_3,\dots,
\theta_k$ for $D^\pm_{m+1}$ and $E_6$).
\def\arraystretch{1.5}
\begin{table}
\begin{equation*}
\begin{array}{||l|l|l|l||}\hline\hline
  \text{Type}\ & f(\theta) & f_1(\theta),\dots, f_n(\theta) & \mbox{} \\ \hline
                                                                      
A_{m+1} & \pm \theta_1^{m+2} + (\theta')^2 &
                                             \theta_1,\dots,\theta_1^m
  & n \geq m \geq 0,\ k \geq 1\\
 \hline D^\pm_{m+1} & \theta_1^2\theta_2 \pm \theta_2^m+ (\theta')^2 &
                                                                 \theta_1,\theta_2,\dots,\theta_2^{m-1}&
n \geq m \geq 3,\  k \geq 2\\
  \hline E_6 & \theta_1^3\pm \theta_2^4 +(\theta')^2 &
                                                                 \theta_1,\theta_2,\theta_2^2,
  \theta_1\theta_2, \theta_1 \theta_2^2& n \geq 5,\ k \geq 2\\
  \hline E_7 & \theta_1^3 + \theta_1\theta_2^3+(\theta')^2 & \theta_1,\theta_2,\theta_2^2,\theta_2^3,\theta_2^4,\theta_1\theta_2  &n\geq 6,\ k\geq 2\\
  \hline E_8 &\theta_1^3+\theta_2^5+(\theta')^2 &\theta_1,\theta_2,\theta_2^2,\theta_2^3,\theta_1\theta_2,\theta_1\theta_2^2,\theta_1\theta_2^3 & n\geq 7, k\geq 2\\
\hline\hline
\end{array}
\end{equation*}
\caption{Classification of \yellow{stable simple} singularities with parametrizations.}\label{table:classification}
\end{table}

The virtue, from the point of view of our analysis, of the
parametrizations in Table~\ref{table:classification} is that the functions $f$ are always
weighted homogeneous, as are the $x_j f_j$ if we consider a joint
homogeneity in $x,\theta.$  We will employ these facts below in our
analysis of the asymptotics.

Which of these singularities appear in ``real-life'' Hamiltonian
systems seems to be an intriguing open question.
We may easily find the fold singularity ($A_2$) arising in
integrable systems: a one-dimensional harmonic oscillator
$$
p=x^2+\xi^2
$$
has a fold singularity at each turning point of the Lagrangian torus
$p=E$ for every $E>0.$  In two dimensions, we may also find fold
singularities in the geodesic flow on convex surfaces of rotation: on
the surface
$$
\big\{ (x, f(x) \cos \theta, f(x) \sin \theta)\colon x \in [a,b],\
\theta \in S^1 \},
$$
the
Clairaut integral constrains the projection of a Lagrangian torus to
be a cylinder lying between two extremal values of the $x$ variable, where the
torus projection has a fold.

More complex singularities seem harder to come by in simple examples
of integrable systems; examples are known, at least numerically, for invariant tori in
nonintegrable settings, however.  For instance, the H\'enon--Heiles
Hamiltonian has been shown to have invariant tori with cusps ($A_3$)
\cite{StVi:95}; Section 5 of \cite{StVi:95} also refers to the
existence of swallowtails in analogous computations for $n=3.$ The
notion of stability employed in Arnol'd's classification is probably
\emph{not} the physically relevant one for KAM systems where we have a
Hamiltonian of the form $\smallabs{\xi}^2 +V(x):$ corners, for
instance, arise naturally and stably in these settings---see
\cite{De:87} and further discussion in \cite{Mo:91}.  Likewise, it is
natural in exploring extremizing sequences of eigenfunctions to
explore the \emph{blowdown singularity}, as this is the (unstable)
singularity to which is associated the extremizing sequence of
spherical harmonics on $S^n.$ We furthermore do not consider
degenerate Lagrangian tori, such as the equatorial orbits on surfaces
of rotation on
which Gaussian beams may concentrate.  We focus here on Arnol'd's
stable simple
singularities merely on the grounds that they are the first natural
case to consider.

\section{The H\"ormander--Melrose theory for $\delta$-Lagrangians}

In this section, we show that $\delta$-Lagrangian distributions can be
obtained as Fourier integrals with symbols in a suitable symbol
class. This is a semiclassical version of the H\"ormander--Melrose
theory (previously worked out in \cite{Alexandrova:Semiclassical} in
the case $\delta=0$), adapted to the case of $\delta$-Lagrangian regularity.

The results in this section are local in nature
and so it suffices to work in Euclidean space.  More precisely, the
results may also be microlocalized: if $B \in \Psi_h(X)$ has compact
microsupport then $Bu_h$ is $\delta$-Lagrangian whenever $u_h$ is (since
we can just replace $A_N$ by $A_N B$ in verifying the oscillatory
testing definition.  Thus, we may always restrict our analysis to
distributions $u_h$ microsupported in arbitrarily small sets.

We introduce \red{for $\delta\in[0,1]$, a} symbol class \blue{consisting of families of smooth
  functions whose higher derivatives satisfy sup norm estimates
  that worsen by powers of $h$:
\begin{multline}S^k_\delta(\RR^n\times\RR^N)=\{a(x,\theta;
  h)\colon\lvert \partial_{(x,\theta)}^\alpha a(x,\theta;h) \rvert \leq
  C_\alpha h^{-k-\delta|\alpha|}\\ \text{for all } \alpha \in \NN^{n+N},\ h \in (0, 1)\}.\end{multline}}

We will use the convention on the semiclassical Fourier transform from
\cite{Zw:12}, with
$$
\F_h u_h(\xi) \equiv \int e^{-ix\xi/h} u_h(x) \, dx.
$$
As it occurs frequently in what follows, we employ the shorthand $\plz$
for ``$+\ep$ for all $\ep>0$.'' We will revert to writing the
definition out in full where important quantities may depend on the
choice of $\ep,$ however.

We will require, in what follows, a sharp version of Sobolev embedding
associated to distributions that are $\delta$-Lagrangian with respect
to the zero section $o \subset T^*\RR^n.$ (Note that such distributions are in
fact exactly the symbols we will be dealing with, since the zero
section is parametrized by the phase function $\phi=0,$ and the
distribution is its own amplitude.
\begin{lemma}\label{lemma:semisobolev}
Let $a(x; h)$ be a
$\delta$-Lagrangian distribution with respect to the zero section.  Then $a \in S_\delta^{\frac{n\delta}2},$ with estimates
depending on only finitely many $\delta$-Lagrangian seminorms.
  \end{lemma}
\blue{Note that Lemma~\ref{lemma:semisobolev} is sharp, as shown by
  the example
 \begin{equation}\label{gaussian}a(x; h)=h^{-\delta/2} e^{-x^2/h^{2\delta}}\end{equation} in
    one dimension.}

  \begin{proof}
    For any \blue{semiclassical family of functions} $u_h,$ let
    $$
    T^\delta_hu_h(\xi)=(2\pi h)^{-n\delta/2} \int u_h(x) e^{-i\xi x/h^{\delta}} \,
    dx
    $$
    denote the semiclassical Fourier transform on scale $h^\delta;$
note that we have scaled $T^\delta_h$ to be unitary, with
    $$
    (T^\delta_h)^{-1}v_h(x)=(2\pi h)^{-n\delta/2} \int v_h(\xi) e^{i\xi x/h^{\delta}} \,
    d\xi.
    $$
    Thus by integration by parts, for all $\alpha$ and $\beta,$
    $$
\xi^\alpha T^\delta_h (h^\delta D_x)^\beta a=T^\delta_h (h^{\delta}
D_x)^{\alpha+\beta} a \in L^2,
$$
uniformly as $h\downarrow 0.$  In particular, then,
$$
\ang{\xi}^{n/2+1} T^\delta_h (h^\delta D_x)^\beta  a \in L^2,
$$
hence by Cauchy--Schwarz applied to the inverse transform
$$\begin{aligned}
\sup\smallabs{(h^\delta D_x)^\beta a} &\leq (2 \pi h)^{-n\delta/2}
\smallnorm{\ang{\xi}^{-n/2-1}}_{L^2} \smallnorm{\ang{\xi}^{n/2+1} T^\delta_h
  (h^\delta D_x)^\beta a}_{L^2} \\ &\leq C_\beta h^{-n\delta/2}
\end{aligned}$$
for all $\beta.$

\end{proof}

Fix a Lagrangian submanifold $\lag\subset T^*\RR^n \simeq \RR^{2n}$ and let $\phi$
be a phase function that locally parametrizes $\lag$ with $N$ phase
variables as described in \S\ref{section:stable}. In particular we
assume that  
\[\lag\cap U=\{(x,\phi_x(x,\theta))\subset \RR^{2n}:\teal{(x,\theta)\in V\textrm{ and }}\phi_\theta(x,\theta)=0 \}\]
where $U\subset{\RR^{2n}}$ \teal{ and $V\subset \RR^{n}\times \RR^N$ are} open and bounded.
 
\blue{Given a symbol $a$ and phase function $\phi,$ we will employ the standard oscillatory integral notation
$$
I(a,\phi)[x]\equiv\int_{\RR^N} a(x,\theta) e^{i\phi(x,\theta)/h} \, d\theta.
$$}%

\begin{proposition}
\label{HM1}
Let $\delta \in [0, 1/2).$
\begin{enumerate}
\item \blue{Let $u_h$ be a $\delta$-Lagrangian distribution with respect to
$\lag,$ with $\|u_h\|_{L^2}=1$ and $\WF_h(u_h)\subset U$.}
For every point $\gamma=(x_0,\xi_0)\in U\cap\lag$, we can find a symbol $a(x,\theta)$ in the class $S^{\frac{N}{2}+\frac{n\delta}2}_\delta(\RR^{n+N})$ such that
$$
u_h=I(a,\phi).
$$
microlocally near $\gamma$.

\item \teal{Conversely, let $a(x,\theta)$ be a symbol in the class $S_{\delta}^{\frac{N}{2}}$ supported in $V$. Then 
	\[u_h=I(a,\phi)\]
	is a $\delta$-Lagrangian distribution $u_h$ with $\WF_h(u_h)\subset U$ and $\|u_h\|_{L^2}$ is bounded.
}
\end{enumerate}
\end{proposition}
\blue{We remark that the discrepancy in symbol orders in the two parts of this
proposition is necessary even in the model case where $\lag$ is the zero-section,
as shown by the example \eqref{gaussian}.}
\begin{proof}
	We closely follow the proof of Theorem~4.4 of \cite{Alexandrova:Semiclassical} and begin by assuming that $\lag$ is transverse to the constant section $\xi=\xi_0$ at $\gamma$. In particular, this implies that we can write 
	\begin{equation}\label{param}\lag\cap U=\{(\partial_\xi H(\xi),\xi)\in \RR^{2n}:x\in W\}\end{equation}
	for some open bounded $W\subset \RR^N$ and some smooth function $H\in\mathcal{C}_b^\infty(W;\RR)$
	which we extend to $\RR^n$.
	The symbols
	\[b_j\equiv x_j-\partial_{\xi_j}H(\xi)\]
	\teal{generate the module of $A\in\Psi_h^{-\infty}$ characteristic to $\lag\cap U$.} 
	\teal{Hence $u_h$ with $\WF_h(u_h)\subset U$ has $\delta$-Lagrangian regularity with respect to $\lag$ if and only if} we have
	\[\|(x-\partial_{\xi}H(hD))^\alpha u_h\|_{L^2} =O(h^{(1-\delta)|\alpha|})\]
	\teal{for all $\alpha$}.
	Taking the semiclassical Fourier transform in $x$ and applying Plancherel, we obtain
	\begin{equation}\label{vreg}\|(-hD-\partial_{\xi}H(\xi))^\alpha \F_h{u_h}\|_{L^2} =O(h^{n/2+(1-\delta)|\alpha|)}).\end{equation}
	Setting 
	\begin{equation}\label{vdef}v_h(\xi)=\red{e^{iH(\xi)/h}}\F_h{u_h}(\xi)\end{equation}
	we obtain
	\[\|\partial^\alpha v_h\|_{L^2}=h^{-|\alpha|}\|(-hD-\partial_\xi H)^\alpha \F_h{u_h}\|_{L^2}=O(h^{n/2-\delta|\alpha|}).\]
	\teal{Hence we have established that for $u_h$ with $\WF_h(u_h)\subset U$ and $\lag$ transverse to the constant section locally parametrized as \eqref{param} that
	\begin{equation}
	\label{dlagcriterion}
	u_h\in L^2 \textrm{ is $\delta$-Lagrangian}\Longleftrightarrow \|\partial_\xi^\alpha (e^{iH/h}\mathcal{F}_h u_h)\|_{L^2}=O(h^{n/2-\delta|\alpha|}) \textrm{ for all }\alpha.
	\end{equation}}
	\teal{Under the assumption that $u_h$ is $\delta$-Lagrangian}, Sobolev embedding yields
	\begin{equation}
	\label{vest}\|\partial^\alpha v_h\|_{L^\infty} =O(h^{n/2-\delta(|\alpha|+n/2)})
      \end{equation}
	and so we have $v_h\in S_\delta^{n(\delta-1)/2}$, and by
        \eqref{vdef}, this shows 
        that we may write $u_h$ as an oscillatory integral
        parametrized by the special phase function $H(\xi)-x\cdot
        \xi$.  \blue{Note that the order of the amplitude, which comes
          out to $n/2+n\delta/2,$ includes a contribution
from the factor of $h^{-n}$ in the inverse Fourier transform.}

	In order to establish the proposition for an arbitrary phase $\phi$ \red{parametrizing a Lagrangian transverse to the constant section satisfying \eqref{param}}, we consider the more general oscillatory integral 
	\[\mathcal{F}_h(I(a,\phi))(\xi)=\int_{\RR^n}\int_{\RR^N} a(x,\theta) e^{i(\phi(x,\theta)-x\cdot \xi)/h} \, d\theta \, dx\]
	for an arbitrary symbol $a\in S^r_\delta(\RR^n\times \RR^N)$.
	
	As in \cite{Alexandrova:Semiclassical}, from the implicit
        function theorem and the nondegeneracy of the phase function
        $\phi$, shrinking $U$ and $W$ if necessary, we can find smooth
        functions
        $\bar{x}\teal{\in\mathcal{C}_b^\infty(W;\RR^n)},\bar{\theta}\in\mathcal{C}_b^\infty(W;\RR^N)$ such
        that for fixed $\xi\in W$, the phase  $$\Phi(x,\theta;\xi)=\phi(x,\theta)- x\cdot \xi$$ 
	is stationary precisely in $(x,\theta)$ at
        $(\bar{x}(\xi),\bar{\theta}(\xi);\xi)$, and this stationary
        point is nondegenerate.
      Furthermore, if $a$ is compactly
        supported close to $(\bar{x}(\xi_0),\bar{\theta}(\xi_0))$,
        then $\mathcal{F}_h(I(a,\phi))$ is $O(h^\infty)$ for
        $\xi\notin W$ by nonstationary phase, and
        $\mathrm{sgn}(\partial^2\Phi)$ can be assumed to be constant
        on the support of $a$.

	For $\xi\in W$ we have the stationary phase expansion
	\[\mathcal{F}_h(I(a,\phi))(\xi)\teal{=} e^{i\Phi(\bar{x}(\xi),\bar{\theta}(\xi);\xi)/h}\sum_{k=0}^{K-1} h^{n/2+N/2+k}(P_{2k}(D)a)(\bar{x}(\xi),\bar{\theta}(\xi))\teal{+R_K(\xi)}\]
	where $P_{2k}$ is a differential operator of order $2k$,
	\[P_0= (2\pi)^{(n+N)/2}|\det(\partial^2\Phi)|^{-1/2}\cdot e^{i\pi\mathrm{sgn}(\partial^2\Phi)/4}\]
	\teal{and
	\[\sup|R_K|\leq C_Kh^{n/2+N/2+K}\!\!\!\!\!\!\!\!\sum_{|\alpha|\leq 2K+n+N+1}\!\!\!\!\!\!\!\!\sup |\partial^\alpha a|=O(h^{-r-\delta+(n/2+N/2+K)(1-2\delta)}).\]
} 
Since $(\pa_\xi H(\xi),\xi)=(\bar{x}(\xi), \xi) \in \lag,$
	we obtain
	\[\partial_\xi \Phi(\bar{x}(\xi),\bar{\theta}(\xi);\xi)=-\pa_\xi H(\xi)\]
	and so by adding a suitable constant to $H$ we may assume that \[\Phi(\bar{x}(\xi),\bar{\theta}(\xi);\xi)=-H(\xi)\] 
	for $\xi\in W$.
	
	\teal{
	Recalling that $\delta < 1/2$, we can choose $K$ sufficiently large so that \[\sup|R_K|=O(h^{-r+n/2+N/2+M(1-2\delta)})\] for arbitrary $M\in\mathbb{N}$, giving 
	\begin{equation}\label{linfest}
	e^{iH/h}\mathcal{F}_h(I(a,\phi))=\sum_{k=0}^{M-1} h^{n/2+N/2+k}(P_{2k}a)(\bar{x}(\xi),\bar{\theta}(\xi))+O_{L^\infty}(h^{-r+n/2+N/2+M(1-2\delta)}).
	\end{equation}
	To estimate the derivatives of $e^{iH/h}\mathcal{F}_h(I(a,\phi))$ we compute
	\[hD_{\xi_k}(e^{iH/h}\mathcal{F}_h(I(a,\phi)))=e^{iH/h}\int_{\RR^n}\int_{\RR^N}(\partial_{\xi_k}H(\xi)-x_k)a(x,\theta)e^{i(\phi(x,\theta)-x\cdot \xi)/h}\, d\theta \,dx.\]
	From the nondegeneracy of the stationary points $(\bar{x}(\xi),\bar{\theta}(\xi);\xi)$, the map $(x,\theta,\xi)\mapsto (\partial_x\Phi,\partial_\theta\Phi,\xi)$ is a local diffeomorphism in a neighbourhood of ${\{(\bar{x}(\xi),\bar{\theta}(\xi),\xi):\xi\in W\}}$.
	As the factor $\partial_{\xi_k}H(\xi)-x_k$ vanishes at
        $(\partial_x\Phi,\partial_\theta\Phi,\xi)=(0,0,\xi)$, Taylor
        expansion gives
	\begin{equation}\label{checking}(\partial_{\xi_k}H(\xi)-x_k)e^{i\Phi/h}=h\left(\sum_{i=1}^n
            b^{(k)}_i(x,\theta,\xi) D_{x_i} +\sum_{j=1}^N
            c^{(k)}_j(x,\theta,\xi) D_{\theta_j}\right)e^{i\Phi/h}\end{equation}
	for $b^{(k)}_i,c^{(k)}_j\in\mathcal{C}^\infty_b(\RR^n\times\RR^N\times \RR^n).$
	Integration by parts in the operator $$L_k=\sum_i b^{(k)}_i D_{x_i}+\sum_j c^{(k)}_j D_{\theta_j}$$ thus shows that $$D_{\xi_k}(e^{iH/h}\mathcal{F}_h(I(a,\phi)))=e^{iH/h}\mathcal{F}_h(I(L_k^Ta,\phi))$$ with $L_k^T$ a first order differential operator and so  $L_k^Ta\in S_\delta^{r+\delta}$.
	By iterating this integration by parts we obtain 
	$$D_{\xi}^\alpha(e^{iH/h}\mathcal{F}_h(I(a,\phi)))=e^{iH/h}\mathcal{F}_h(I(L^\alpha a,\phi))$$
	where $L^\alpha$ is a differential operator of order $\alpha$, only involving differentiation in $(x,\theta)$, and with coefficients smooth in $(x,\theta,\xi)$. By utilising \eqref{linfest}, we obtain 
		\begin{equation}
		\label{linfest2}
		\|\partial_\xi^\alpha (e^{iH/h}\mathcal{F}_h I(a,\phi))\|_{L^\infty}=O(h^{-r+n/2+N/2-\delta|\alpha|}).
		\end{equation}
	Equation \eqref{linfest2} implies that  $e^{iH/h}\mathcal{F}_h(I(a,\phi))\in S^{r-n/2-N/2}_\delta$, and so from \eqref{linfest} and a semiclassical analogue of \cite[Proposition~1.1.10]{Hormander1} we deduce the expansion
	\begin{equation}
	\label{asymp}
	e^{iH/h}\mathcal{F}_h(I(a,\phi))\sim \sum_{k=0}^\infty h^{n/2+N/2+k}(P_{2k}a)(\bar{x}(\xi),\bar{\theta}(\xi))
	\end{equation}
	in the sense that
	\begin{equation}
	e^{iH/h}\mathcal{F}_h(I(a,\phi))-\sum_{k=0}^{M-1} h^{n/2+N/2+k}(P_{2k}a)(\bar{x}(\xi),\bar{\theta}(\xi))\in S_\delta^{r-n/2-N/2-M(1-2\delta)}.
	\end{equation}
	As $\mathcal{F}_h I(a,\phi)$ is $O(h^\infty)$ outside the bounded set $W$, we can combine  \eqref{asymp} and \eqref{dlagcriterion} to show that $I(a,\phi)$ has $\delta$-Lagrangian regularity and is bounded in $L^2$, proving part (2) of the proposition in the case where $\lag$ is transverse to the constant section. 
}

	We now \red{complete the proof of part (1), under the same transversality assumption}. The idea is to use the expansion \eqref{asymp} to construct a symbol $a(x\teal{,\theta})$ such that $v_h=e^{iH/h}\mathcal{F}_h(I(a,\phi))+O_{S_\delta}(h^\infty)$, where $v_h$ is as in \eqref{vdef}.
	\teal{We write $\psi(x,\theta)=\partial_x\phi(x,\theta)$. This function is smooth in a neighbourhood $V$ of $(\bar{x}(\xi_0),\bar{\theta}(\xi_0))$ and satisfies  $\psi(\bar{x}(\xi),\bar{\theta}(\xi))=\xi\in \RR^n$ as $\Phi$ is stationary in $(x,\theta)$ at $(\bar{x}(\xi),\bar{\theta}(\xi))$}.

	We begin by taking
	\[a_0=(2\pi h)^{-(n+N)/2}\left(|\det(\partial^2\Phi)|^{1/2}\cdot e^{-i\pi\mathrm{sgn}(\partial^2\Phi)/4}v_{\teal{h}}\right)\circ \psi\ \]
	for $(x,\theta)$ near $(\bar{x}(\xi_0),\bar{\theta}(\xi_0))$ and cutting off smoothly away from $V^c$, we have $a_0\in S_\delta^{(n\delta+N)/2}$ and  \teal{by truncating the expansion\eqref{asymp} after the leading term, we obtain}
	\[e^{iH/h}\mathcal{F}_h(I(a_0,\phi))-v_h\teal{\in S_\delta^{n(\delta-1)/2-(1-2\delta)}}.\]
	Proceeding iteratively, we can construct a sequence of symbols 
	$${a_k\in S_\delta^{\teal{(n\delta+N)/2-(1-2\delta)k}}}$$ supported in $V$ such that 
	\[e^{iH/h}\mathcal{F}_h\bigg(I\big(\sum_{k=0}^{\teal{l-1}} a_k h^{k(1-2\delta)},\phi\big)\bigg)-v_h\teal{\in S_\delta^{n(\delta-1)-l(1-2\delta)}}.\]
	Borel summation then yields a total symbol $a\in S_\delta^{(n\delta+N)/2}$ with 
	\[e^{iH/h}\mathcal{F}_h(I(a,\phi))-v\teal{\in S_\delta^{-\infty}}\]
	which allows us to conclude that 
	\[u_h=I(a,\phi)\]
	microlocally near $(x_0,\xi_0)$, with $a$ in the required symbol class.

	\teal{It remains to establish parts (1) and (2) of the Proposition in}  the case where $\lag$ is not transverse to the
        constant section $\xi=\xi_0$ at $\gamma=(x_0,\xi_0)$. We
        proceed as in \cite{Alexandrova:Semiclassical} and apply a
        symplectic transformation to reduce to the transverse case as follows.
	
	We can choose our coordinates $x=(x',x'')$ and $\xi=(\xi',\xi'')$ in $\RR^{k}\times \RR^{n-k}$ so that the tangent space $T_\gamma\lag$ takes the form
	\[T_\gamma\lag=\{(0,x'';\xi',Bx''):x''\in \RR^{n-k},\xi'\in\RR^k\}\]
	where $B$ is a symmetric matrix. Here we have identified $T_\gamma\lag$ with a $n$-dimensional subspace of $\RR^n\times\RR^n$ in the natural way. If $B$ were invertible, then this tangent space would be transverse to the constant section, so we choose a diagonal $(n-k)\times (n-k)$ matrix $D$ such that $B+D$ is nonsingular.
	
	Then the transformed Lagrangian
	\[\tilde{\lag}=\{(0,x'';\xi',(B+D)x''):x''\in \RR^{n-k},\xi'\in\RR^k\}\]
	is transverse to the constant section through
        \blue{$\tilde{\gamma}\equiv (x_0,\xi_0+Dx_0'')$}
	and is parametrized by the phase function
	\[\tilde{\phi}(x,\theta)=\phi(x,\theta)+\frac{1}{2}Dx''\cdot x''.\]
	
	Taking \blue{$A_j\in \Psi_h$} characteristic to $\lag$ and
        \blue{compactly} microlocalized near $\gamma$, partial Lagrangian regularity implies
	\[\left(\prod_{j=1}^m e^{iDx''\cdot x''/2h}A_je^{-iDx''\cdot x''/2h}\right)e^{iDx''\cdot x''/2h}u_h=O_{L^2}(h^{(1-\delta)m}) \]
	\teal{for $L^2$-normalised $u_h$ with partial Lagrangian regularity with respect to $\lag$.}
	The operators
	\[B_j=e^{iDx''\cdot x''/2h}A_je^{-iDx''\cdot x''/2h}\]
	are shown in \cite{Alexandrova:Semiclassical} to be semiclassical pseudodifferential operators that are compactly microlocalized near $\tilde{\gamma}$ with principal symbols \begin{equation}\label{symbolrel}
	\sigma(B_j)(x,\xi)=\sigma(A_j)(x,\xi-Dx'')
\end{equation}
	which are characteristic to $\tilde{\lag}$.
	
	From \teal{part (1)} of the proposition \teal{in the case} where $\lag$ is
        transverse to the constant section, it follows that we can
        find a symbol $a\in S_\delta^{N/2+n\delta/2}$ with  
	\[ e^{iDx''\cdot x''/2h}u_h= I(a,\tilde{\phi})\]
	microlocally near $\tilde{\gamma}$ and so we can conclude that
	\[u_h= I(a,\phi)\]
	microlocally near $\gamma$. This completes the proof of part (1) of the proposition.
	
	\teal{Similarly, if $u_h$ is given by $I(a,\phi)$ for $a\in S_\delta^{\frac{N}{2}}$, then $e^{iDx''\cdot x''/2h}u_h=I(a,\tilde{\phi})$. From part (2) of the proposition in the case where $\lag$ is transverse to the constant section, it follows that $e^{iDx''\cdot x''/2h}u_h$ is an $L^2$-bounded $\delta$-Lagrangian distribution with respect to \red{$\tilde\lag$}. As such, we have $$\left(\prod_{j=1}^m B_j\right)e^{iDx''\cdot x''/2h} u_h=O_{L^2}(h^{(1-\delta)m})$$ for any collection of $B_j\in \Psi_h$ characteristic to $\tilde \lag$ and compactly microlocalized near $\tilde \gamma$. In particular, by \eqref{symbolrel} this is true for $B_j=e^{iDx''\cdot x''/2h}A_j e^{-iDx''\cdot x''/2h}$ where $A_j\in\Psi_h$ is characteristic to $\lag$ and compactly microlocalized near $\gamma$, and we obtain
	$$\left(\prod_{j=1}^m A_j\right) u_h=O_{L^2}(h^{(1-\delta)m})$$
	for arbitrary such $A_j$, which completes the proof of part (2) of the proposition.
	}
\end{proof}

\blue{In the case that the Lagrangian is \emph{projectable} onto the base manifold
i.e., that the projection map is a diffeomorphism,} we can parametrize it using a phase function $\phi$ with
$0$ phase variables, and a simpler argument establishes the result in
Proposition \ref{HM1} without the restriction that $\delta< 1/2$.

\begin{proposition}
\label{HM2}
\red{Let $\delta \in [0,1]$ and suppose $u_h$ is a semiclassical distribution with $\delta$-Lagrangian regularity with respect to an arbitrary Lagrangian $\lag\subset T^*X$.}
For every point $\gamma=(x_0,\xi_0)\in U\cap\lag$ at which $\lag$ is projectable and parametrized by the phase function $\phi(x)$, we can find a symbol $a(x)$ in the class $S^{\frac{n\delta}2}_\delta(\RR^{n})$ such that
\blue{$$
u_h(x)=a(x) e^{i\phi(x)/h}
$$}%
microlocally near $\gamma$.
\end{proposition}
\begin{proof}
From the assumptions on $\lag$, we can find a bounded open set $W\subset \RR^n$ with 
\[\lag\cap U=\{(x,\partial_x\phi(x))\in \RR^{2n}:x\in W\}.\]
The symbols
\[b_j:=\xi_j-\partial_{x_j}\phi\]
are then characteristic to $\lag\cap U$ and by partial Lagrangian regularity, we have
\[\|(hD-\partial_{x}\phi)^\alpha u_h\|_{L^2} =O(h^{(1-\delta)|\alpha|}).\]
Setting 
\[a(x)=u_h(x)e^{-i\phi(x)/h}\]
we obtain
\[\|\partial^\alpha a\|_{L^2}=h^{-|\alpha|}\|(hD-\partial_x \phi)^\alpha u_h\|_{L^2}=O(h^{-\delta|\alpha|}).\]
Sobolev embedding yields
\[\|\partial^\alpha a\|_{L^\infty} =O(h^{-\delta(|\alpha|+n/2)})\]
and so we have $a\in S_\delta^{n\delta/2}(\RR^n)$.
\end{proof}

More generally, we now show that we can also obtain Fourier integral representations
for $\delta$-Lagrangian distributions with $\delta \geq 1/2$, provided
we restrict ourselves to a 
particular class of phase functions.

As in the proof of Proposition $\ref{HM1}$, it suffices to treat the case where $\lag\cap U$ is transverse to the constant section $\xi=\xi_0$ at $\gamma$. Under this assumption, we can locally parametrize our Lagrangian as
\[\lag\cap U=\{(\pa_\xi H(\xi),\xi):\xi\in W\}\] 
for some smooth function $H$ \blue{and open set $W$.} If the point $\gamma\in \lag$ does not
lie on the zero section, then we can always obtain this transversality
condition by choosing coordinates on the base space appropriately
\cite[p.~102]{Grigis-Sjostrand}. 
After choosing such coordinates, one possible choice of phase function to locally parametrize $\teal{\lag}$ is
\[\phi(x,\theta)=x\cdot \theta-H(\theta).\]
For this particular choice of phase function we have a simpler
argument to arrive at the analogous result to Proposition \ref{HM1},
valid for all $\delta\in[0,1]$.  Recall that $X$ denotes a smooth $n$-manifold.
\begin{proposition}
 \label{HM.explicitphase}
\begin{enumerate}
Let $\delta \in [0,1].$
\item
Suppose $u_h$ is a semiclassical distribution with $\delta$-Lagrangian regularity with respect to \teal{an arbitrary Lagrangian} $\lag\subset T^*X$.
For every point $\gamma\in \lag$, we can choose local coordinates on
$X$ and find a symbol $a(\theta)$ in the class $S_\delta^{\frac
  n2+\frac{n\delta}2}(\RR^{2n})$ and a  function $\psi\in \mathcal{C}^\infty(\RR^n)$ such that
\blue{\[u_h(x)=I(a,\phi)[x]=\int_{\RR^n}a(\theta)e^{i(x\cdot\theta-H(\theta)-\psi(x))/h}\, d\theta\]}%
microlocally near $\gamma$.

If $\gamma$ does not lie on the zero section of $T^*X$ then we can take $\psi=0$.

\item \teal{Conversely, for a Lagrangian locally parametrized as
	$$\lag\cap U=\{(\partial_\xi H(\xi),\xi):\xi \in W\} $$
	and $a\in S_\delta^{\frac{n}{2}}$ supported in
        $W$, $$u_h(x)=I(a,\phi)[x]=\int_{\RR^n}
        a(\theta)e^{i(x\cdot\theta-H(\theta))/h}\, d\theta$$
        determines a $\delta$-Lagrangian distribution with respect to
        $\lag$ 
with $\WF_h(u_h)\subset U$; moreover $\|u_h\|_{L^2}$ is bounded.}

\end{enumerate}
\end{proposition}
\begin{proof}
	\teal{We begin by proving part (1) of the proposition.}
	We may assume without loss of generality that $u_h$ is compactly microlocalized in a neighbourhood $U$ of $\gamma$ by applying a microlocal cutoff.
	
	First we suppose that $\gamma$ does not lie in the zero
        section. Then again by choosing coordinates on the base space appropriately we can locally parametrize our Lagrangian $\lag$ as
	\[\lag\cap U=\{(\pa_\xi H(\xi),\xi):\xi\in W\}\]
	in induced canonical coordinates $(x,\xi)$, for some $H\in\mathcal{C}^\infty(\RR^n)$ where $U\subset T^*X$ and $W\subset \RR^n$ are open and bounded.
	Setting
	\begin{equation}\label{aurel}a=(2\pi h)^{-n}\mathcal{F}_hu_h\cdot e^{iH/h},\end{equation}
	semiclassical Fourier inversion immediately yields the sought Fourier integral representation, and from \eqref{vest}, it follows that 
	\[\|\partial^\alpha a\|_{L^\infty} =O\left(h^{-\frac{n(1+\delta)}{2}-\delta|\alpha|}\right)\]
	as required.
	
	On the other hand, if $\gamma=(x_0,\xi_0)$ does lie in the zero section, we consider the distribution $\tilde{u}_h=e^{i\psi/h}u_h$ for an arbitrary smooth real-valued $\psi$ with $\psi'(x_0)\neq 0$.
	Since $u_h$ is $\delta$-Lagrangian with respect to $\lag$, for any collection of operators $A_j\in \Psi_h^{-\infty}$ that are characteristic to $\lag$ we have the iterated regularity estimate
	\[\left\| \left(\prod_{j=1}^N e^{i\psi/h}A_je^{-i\psi/h}\right)\tilde{u}_h\right\|_{L^2} =O(h^{(1-\delta)N}). \]
	By Egorov's theorem, each of the operators 
	\[\tilde{A}_j=e^{i\psi/h}A_je^{-i\psi/h}\]
	is itself a semiclassical pseudodifferential operator, with principal symbol
	\[\sigma(\tilde{A}_j)=\sigma(A_j)(x,\xi-\psi'(x).)\]
	It follows that $\tilde{u}_h$ enjoys $\delta$-Lagrangian regularity with respect to
	\[\tilde{\lag}=\{(x,\xi-\psi'(x)):(x,\xi)\in\lag\}\]
with $\tilde{\gamma}\equiv\gamma+(0,\psi'(x))$ not lying in the zero section.
We can now choose coordinates on the base space $X$ such that in the
associated canonical coordinates, the Lagrangian $\tilde{\lag}$ is
locally parametrized near $\tilde{\gamma}$ by 
\[\tilde{\lag}\cap \tilde{U}=\{(\pa_\xi H(\xi),\xi):\xi\in \tilde{W}\} \]
for some $H\in\mathcal{C}^\infty(\RR^n)$ \blue{and for some open set $\tilde{W}$.}%
We can now treat $\tilde{u}_h$ as was done for $\gamma$ off the zero section, obtaining the oscillatory integral representation
\[\tilde{u}_h(x)=\int_{\RR^n} a(\theta)e^{i(x\cdot\theta-H(\theta))/h} \, d\theta\]
microlocally near $\tilde{\gamma}$ and consequently
\[u_h(x)=\int_{\RR^n} a(\theta)e^{i(x\cdot\theta-H(\theta)-\psi(x))/h} \, d\theta\]
microlocally near $\gamma$.

\teal{Part (2) of the proposition follows immediately from \eqref{aurel} and \eqref{dlagcriterion}.}
\end{proof}

\subsection{Improvements for quasimodes}

We now additionally assume that the $\delta$-Lagrangian distribution $u_h$ satisfies
\[\|Pu_h\|_{L^2} =O(h)\] for a semiclassical pseudodifferential operator
$P$ of real principal type, with principal symbol $p$
\blue{characteristic to $\lag,$ i.e., vanishing on it.} Note that this hypothesis can be localized, as if $B$ is a
pseudodifferential operator with compact microsupport, then we also have
\[\|P B u_h\|_{L^2} =O(h).\]

Under the hypotheses that $u_h$ is such a quasimode, we will obtain an improvement to Proposition \ref{HM1} and
Proposition \ref{HM.explicitphase}. The first step is obtaining a
mixed iterated regularity estimate.

\begin{lemma}
	\label{HM.mixed}
	Suppose $u_h$ is a compactly microlocalized $\delta$-Lagrangian distribution with respect to $\lag$ that additionally satisfies
	\blue{\[\|Pu_h\|_{L^2}=O(h),\quad \|u_h\|_{L^2}=1\]
	where $P$ is a semiclassical pseudodifferential operator characteristic to $\lag$.}
	Then for any $\epsilon>0$, $u_h$ enjoys the mixed iterated regularity estimate
	\begin{equation}\label{PA}\|PA_1\ldots A_N u_h\|_{L^2}=O(h^{N(1-\delta)+1-\epsilon})\end{equation}
	for any $A_j\in \Psi^{-\infty}$ characteristic to $\lag$.
\end{lemma}
\begin{proof}
We have 
\begin{equation}
\label{commutator}
PA_1\ldots A_N u_h=A_1\ldots A_N Pu_h+O(h^{N(1-\delta)+1}) \end{equation}
as each commutator $[P,A]$ has $O(h)$ principal symbol characteristic to $\lag$.
We now proceed inductively to show that
\begin{equation}
\label{ind.iterated}
\|A_1\ldots A_NP u_h\|_{L^2}= O(h^{N(1-\delta)+1-2^{-k}})
\end{equation}
for every non-negative integer $k$. For $k=0$, \eqref{ind.iterated} \red{follows from \eqref{commutator}, $\delta$-Lagrangian regularity and $L^2$-boundedness of $P$}.
Now if we have \eqref{ind.iterated} for a particular $k$ and any collection of characteristic operators, we can compute
\begin{align*}
\|A_1\ldots A_N Pu_h\|_{L^2}^2 &= |\langle A_N^* \ldots A_1^*A_1 \ldots A_N Pu_h,Pu_h \rangle|\\
&= O(h^{2N(1-\delta)+1-2^{-k}})\cdot O(h)\\
&= O(h^{2N(1-\delta)+2-2^{-k}}).
\end{align*}
Taking square roots completes the induction and \yellow{and using
  \eqref{commutator} once more proves \eqref{PA}.}
\end{proof}

\begin{proposition}\label{prop:quasi}
Let $\delta \in [0, 1/2).$  Suppose $u_h$ satisfies the assumptions of
Lemma \ref{HM.mixed} and that $P$ has real-valued principal symbol $p$
satisfying $|\partial p|\neq 0$ on $p^{-1}(0)$. For every point
$\gamma=(x_0,\xi_0)\in U\cap \lag$, we can find a symbol $a(x,\theta)$
in the class $S_\delta^{\frac N2+\frac{(n-1)\delta}{2}\plz}(\RR^{n+N})$ such that 
\blue{\[u_h=I(a,\phi)\]}%
microlocally near $\gamma$.
\end{proposition}
\begin{proof}
Following the proof of Proposition \ref{HM1}, we can assume that $\lag$ is transverse to the constant section at $\gamma$. It then suffices to prove the estimate
\[\|\partial^\alpha v_h\|_{\infty} =O(h^{n/2-\delta(|\alpha|+(n-1)/2)-\ep}),\]
where $v_h(\xi)=\F_h{u_h}(\xi)e^{iH(\xi)/h}$, improving on \eqref{vest} by a factor of $h^{\delta/2-\ep}$. 

We do this by computing
\begin{align*}
P(x-\partial_\xi H(hD))^\alpha u_h&=  P(x-\partial_\xi H(hD))^\alpha \mathcal{F}_h^{-1} (e^{-iH/h}v_h) \\
&=   P \mathcal{F}_h^{-1} (e^{-iH/h}(-hD)^\alpha v_h).
\end{align*}

From Lemma \ref{ind.iterated} and Plancherel's theorem, it follows that
\begin{equation}
\label{v.l2}
\|QD^\alpha v_h\|_{L^2}=O(h^{\frac{n}{2}-\delta|\alpha|+1-\epsilon})
\end{equation}
where $Q=e^{iH/h}\mathcal{F}_hP\mathcal{F}_h^{-1} e^{-iH/h},$ with the exponential functions being regarded as multiplication operators. The principal symbol of $Q$ is given by 
\[q(x,\xi)=\sigma(\mathcal{F}_hP\mathcal{F}_h^{-1})(x,\xi+\partial_x H)=p(-\xi+\partial_x H,x)\]
from Egorov's theorem, so $Q$ is characteristic to the zero section. 
As $P$ was of real principal type, and characteristic to the Lagrangian $\lag$ which is locally projectable in $\xi$, we have $\partial_x p \neq 0$ and so $\partial_\xi q \neq 0$.
Reordering indices, we can assume
\[q(x,\xi)=e(x,\xi)(\xi_1-b(x,\xi')).\]
with $e,b\in \mathcal{C}^\infty$ and $e(x_0,\xi_0)\neq 0$, where we have split $\xi=(\xi_1,\xi')$.
By shrinking the initial microlocal cutoff of $u_h$ if necessary, the local ellipticity of $e$ together with \eqref{v.l2} implies
\begin{equation}\label{evolution}\|(hD_{x_1}-b(x,hD'))D^\alpha v_h\|_{L^2}= O(h^{\frac{n}{2}-\delta|\alpha|+1-\epsilon}).\end{equation}
Recall that we may microlocalize $u_h$ as finely as we like at the
outset without affecting the hypotheses of this proposition, hence we
assume without loss of generality that $v_h=O(h^\infty)$ outside a small
neighborhood of $\xi_0.$  Consequently
\eqref{evolution} together with \cite[Lemma~7.11]{Zw:12}  implies
\[\|D^\alpha v_h(x_1,\cdot)\|_{L^2(\RR^{n-1})}=O(h^{\frac{n}{2}-\delta|\alpha|-\epsilon}).\]
Again using the fact that $v$ is compactly supported modulo residual terms, Sobolev embedding in the remaining $n-1$ variables yields
\begin{equation}
\label{improved.aest}
\|\partial^\alpha v_h\|_{L^\infty} =O(h^{\frac{n}{2}-\frac{\delta(n-1)}{2}-\delta|\alpha|-\epsilon})
\end{equation}
as required.

\end{proof}

As in Proposition \ref{HM2}, we have a simpler argument in the case
that the Lagrangian is projectable onto $X$, that parametrizes $\lag$
using a phase function $\phi$ with $0$ phase variables.

\begin{proposition}
	\label{HM2.quasi}
	Suppose $u_h$ satisfies the assumptions of Lemma
        \ref{HM.mixed} and that $P$ has  real-valued principal symbol
        $p$ satisfying $|\partial p|\neq 0$ on $p^{-1}(0)$. For every
        point $\gamma=(x_0,\xi_0)\in \lag\cap U$ at which $\lag$ is
        projectable and parametrized by the phase function $\phi(x)$,
        we can find a symbol $a(x)$ in the class
        $S_\delta^{\frac{(n-1)\delta}{2}\plz}(\RR^{n})$ and a function
       $\phi\in\mathcal{C}^\infty(\RR^n)$ such that
\blue{	\[u_h(x)=a(x)e^{i\phi(x)/h}\]}%
	microlocally near $\gamma$. 
\end{proposition}
\begin{proof}
Choosing $U\subset \RR^{2n}$ a small neighbourhood of $\gamma$ with $\lag\cap U$ projectable, we write
\[\lag \cap U=\{(x,\partial_x\phi(x))\in \RR^{2n}:x\in W\}\]
for a bounded open set $W$. The symbols
\[b_j=\xi_j-\partial_{x_j}\phi\]
are then characteristic to $\lag\cap U$
and by Lemma \ref{HM.mixed} we have
\[\|P(hD-\partial_x\phi)^\alpha u_h\|_{L^2}=O(h^{(1-\delta)|\alpha|+1-0}).\]
Taking $a=u_h e^{-i\phi/h}$ as in the proof of Proposition \ref{HM2}, it follows that
\begin{equation}
\label{eq:projectable.quasimode.estimate}
\|Pe^{i\phi/h}D^\alpha a\|_{L^2}=O(h^{-\delta|\alpha|}).
\end{equation}
As $P$ is of real principal type and is characteristic to the Lagrangian $\lag$, which is locally projectable, we have $\partial_\xi p \neq 0$ and by reordering indices we can write $p$ in the form
\[p(x,\xi)=e(x,\xi)(\xi_1-b(x,\xi'))\]
with $e,b\in\mathcal{C}^\infty$ and $e(x_0,\xi_0)\neq 0$, where we have split $\xi=(\xi_1,\xi')$.
The local ellipticity of $e$ and \eqref{eq:projectable.quasimode.estimate} together show that
\[\|(hD_{x_1}-b(x,hD'))e^{i\phi/h}D^\alpha a\|_{L^2}=O(h^{-\delta|\alpha|+1-0}).\]
As $u_h$ can be assumed to be $O(h^\infty)$ outside a small neighbourhood of $x_0$, we can apply \cite[Lemma~7.11]{Zw:12} once again to obtain
\[\|D^\alpha a(x_1,\cdot )\|_{L^2(\RR^{n-1})}=O(h^{-\delta|\alpha|-0}).\]
Sobolev embedding in the remaining $n-1$ variables yields
\[\|D^\alpha a\|_{L^\infty} = O(h^{-\frac{\delta(n-1)}{2}-\delta|\alpha|-0})\]
as required.
\end{proof}
As in Proposition $\ref{HM.explicitphase}$, we may also dispense with the condition that
$\delta<1/2$ if we specialize to a simple class of phase functions.
\begin{proposition}
  Suppose $u_h$ satisfies the assumptions of Lemma \ref{HM.mixed} and
  that $P$ has real-valued principal symbol $p$ satisfying
  $|\partial p|\neq 0$ on $p^{-1}(0)$. For every point
  $\gamma=(x_0,\xi_0)\in \lag\cap U$, we can choose local coordinates
  on $X$ and find a symbol $a(\theta)$ in the class
  $S_\delta^{\frac n2 +\frac{(n-1)\delta}2 \plz}(\RR^{n})$
  and \blue{functions $H,\ \psi\in\mathcal{C}^\infty(\RR^n)$} such that
\blue{\[u_h(x)=I(a,\phi)[x]=\int_{\RR^n}a(\theta)e^{i(x\cdot\theta-H(\theta)-\psi(x))/h}\,
    d\theta\]}%
microlocally near $\gamma$. If $\gamma$ does not lie on the zero section of $T^*X$ then we can take $\psi=0$.
\end{proposition}
\begin{proof}
As in the proof of Proposition \ref{HM.explicitphase}, we begin by microlocalizing $u_h$ to a neighbourhood $U$ of $\gamma$, making the assumption that $\gamma$ does not lie on the zero section, and choosing canonical coordinates so that $\lag$ is locally projectable in $\xi$.
The estimate \eqref{improved.aest} then immediately implies
\[\|\partial^\alpha a\|_{L^\infty} =O(h^{-\frac{n}{2}-\frac{\delta(n-1)}{2}-\delta|\alpha|-\epsilon})\]
as required.

If $\gamma$ lies on the zero section, then we can proceed as in the
end of the proof of Proposition \ref{HM.explicitphase}, noting that
$\tilde{u}_h$ will necessarily be an $O(h)$ quasimode for for the
conjugated operator $\tilde{P}=e^{i\psi/h}Pe^{-i\psi/h}$. 
\end{proof}

\section{Duistermaat's degenerate stationary phase and $L^\infty$
  estimates below threshold}
\label{sec:below.threshold}

\yellow{Let $\lag$ be a Lagrangian with a stable simple singularity
and $u_h$ be a $\delta$-Lagrangian distribution with respect to
$\lag,$ microsupported in a small neighborhood of the singularity
in question and with $\norm{u_h}_{\red{L^2}}=1.$  Assume that $\delta\leq \delta_0$ where $\delta_0$ is
the threshold for the singularity type (as listed in Table~\ref{table:intro-orders}).
Fix a phase function $\phi_0=x\cdot \theta-H(\theta)$ parametrizing the stable simple
singularity.  By Propositions~\ref{HM1} and \ref{prop:quasi} (which we may apply since all
thresholds $\delta_0$ in question are less than $1/2$),
$$
u_h(x)= I(a_0,\phi_0)[x]
$$
where $a_0 \in S_\delta^{\frac n 2+\frac{n\delta}2}(\RR^{2n})$ in general or $a_0 \in
S_\delta^{\frac n2+\frac{(n-1)\delta}2+0}(\RR^{2n})$ if $u_h$ satisfies an
equation as described the latter proposition.  By the classification
of stable simple singularities, there is
$\phi(x,\theta)=\sum x_j f_j(\theta) +f(\theta)$ chosen from
Table~\eqref{table:classification} above that is locally equivalent to
$\phi_0$ in the sense that
$$
\phi_0(x,\theta)=\phi(x',\theta')+\psi(x')
$$
for some local fiber-preserving diffeomorphism
$$
(x,\theta) \mapsto (x(x'), \theta(x',\theta'))
$$
and some $\psi \in \CI.$
We thus change coordinates in the integral $I(a_0,\phi_0)$ from $\theta$ to
$\theta'$ and note that pullback under this coordinate change leaves
$a$ in the same symbol class.  This results in an integral of the form
$I(a, \phi)$ with $a$ a symbol of the same type, times an overall phase factor
$e^{i\psi/h},$ all pulled back by a local diffeomorphism in $x.$  Consequently
in order to prove
Theorem~\ref{theorem:main}, it suffices} to show that an oscillatory
integral with one of the phase functions in
Table~\eqref{table:classification} and with an amplitude lying in
$S^{k/2}_\delta$ (where $k$ is the number of phase variables) \teal{is} $O_{L^\infty}(h^{-\kappa});$ here we have multiplied
through by $h^{n\delta/2}$ resp.\ $h^{((n-1+\ep)\delta/2}$ in the two
cases of a general $\delta$-Lagrangian or a quasimode in order to
eliminate the $\delta$-dependence of the symbol order.  In other
words, pulling out an explicit factor of $h^{-k/2}$ as part of the
normalization of the integral, it will suffice to prove the following:
\begin{theorem}\label{theorem:sufficient}
  Let $$I(x) = h^{-k/2} \int_{\RR^k} a(x,\theta) e^{i\phi/h} \, d\theta,$$
  where $\phi$ is one of the phase functions arising in
  Table~\ref{table:classification}, and where
  $$
a \in S^0_\delta.
$$
For $\delta \in [0, \delta_0],$ where $\delta_0$ is the threshold
value listed in Table~\ref{table:intro-orders}, \blue{there exists $C$ such
that for all $h \in (0,1),$}
$$
\norm{I(x)}_{L^\infty} \leq Ch^{-\kappa}
$$
where $\kappa$ is the order of the caustic listed in Table~\ref{table:intro-orders}.
\end{theorem}

A novelty of the approach here is that
we are unable to employ the Malgrange Preparation Theorem/Mather
Division Theorem as in the
classic treatments with $\delta=0$ \cite[Lemma~2.1.4,
Equation (4.1.3)]{Duistermaat:Oscillatory} and
\cite[Theorem~9.1]{Guillemin-Sternberg1}: the trouble is that the use
of the Preparation Theorem costs numerous derivatives which are hard
to keep track of, and each of these derivatives hitting the amplitude
costs us $h^\delta.$  Since we are trying to obtain a cruder result
(estimates rather than full asymptotics) we are able to use simpler
and more robust methods.  We now describe the method of proof.

\yellow{Recall that we work only with the simple stable caustics in Arnol'd's
classification.  Depending on the overall dimension, each of these
caustics can have an
``equisingularity manifold'' along which the form of the singularity
of the projection is unchanged. In the model cases under discussion
this arises just because the phase is independent of
some of the $x$ variables, in particular, of all the $x_j$ variables
corresponding to vanishing $f_j$ in Table~\ref{table:classification}.
Let $k_0$ be the largest $j$ with $f_j\neq 0,$ so that
$j=m$ for $A_{m+1},$ $j=m$ for $D^\pm_{m+1},$ and $j =5,6,7$ for $E_6,$
$E_7,$ and $E_8$ respectively.  Then near any point in the manifold $\{(x_1,\dots,
x_{k_0})=0\}\subset \RR^n \times \RR^k$ the singularity is of the same type as at the origin,
hence the term ``equisingularity manifold.''  (A definition of
equisingularity applicable
in more general cases, but not needed here, is as the
set of points where the germ of the phase is equivalent (via germs of
mappings) to the singularity at a given point---see \cite[p.243]{Duistermaat:Oscillatory}.)}

\yellow{Away from the equisingularity manifold, the Lagrangian
  projection will have a different singularity than that near the
  origin; the latter singularity is said to be \emph{subordinate} to the one
  at the origin (see \cite[p.255]{Duistermaat:Oscillatory}).  Arnol'd's
  classification comes with a characterization of subordinate
  singularities, encapsulated in the 
``subordination diagram'' of caustics given in
Figure~\ref{fig:subordination}.  Arrows point from singularities to
those types that may possibly arise in a small neighborhood of the origin
in the complement of the equisingularity manifold $\{(x_1,\dots x_{k_0})=0\}.$}
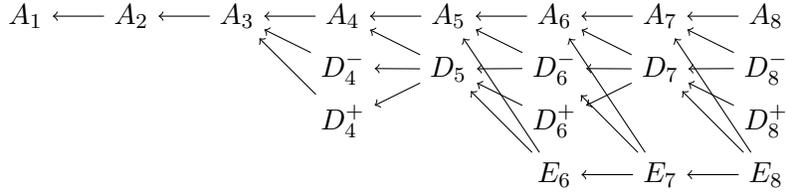
\begin{figure}
\begin{tikzpicture}
 \matrix[matrix of math nodes,column sep={40pt,between origins},row
    sep={20pt,between origins}] (s)
  {
    |[name=A1]| A_1 &|[name=A2]| A_2 &|[name=A3]| A_3 &|[name=A4]| A_4 &|[name=A5]| A_5 &|[name=A6]| A_6 &|[name=A7]| A_7 &|[name=A8]| A_8  \\
    & & & |[name=D4m]| D_4^- & |[name=D5]| D_5 & |[name=D6m]| D_6^- & |[name=D7]| D_7 & |[name=D8m]| D_8^-\\
    & & & |[name=D4p]| D_4^+& & |[name=D6p]| D_6^+& & |[name=D8p]| D_8^+\\
    & & &  & & |[name=E6]| E_6 & |[name=E7]| E_7 & |[name=E8]| E_8\\             
  };
  \draw[->]
  (A2) edge (A1)
  (A3) edge (A2)
  (A4) edge (A3)
  (A5) edge (A4)
  (A6) edge (A5)
  (A7) edge (A6)
      (A8) edge (A7)
  (D4m) edge (A3)
    (D4p) edge (A3)
    (D5) edge (D4m)
        (D5) edge (D4p)
        (D5) edge (A4)
        (D6m) edge (D5)
        (D6p) edge (D5)
        (D6m) edge (A5)
        (E6) edge (A5)
        (E6) edge (D5)
    (D7) edge (D6m)
        (D7) edge (D6p)
        (D7) edge (A6)
        (E7) edge (E6)
        (E7) edge (A6)
        (E7) edge (D6m)
        (D8m) edge (D7)
        (D8p) edge (D7)
        (D8m) edge (A7)
        (E8) edge (E7)
        (E8) edge (A7)
        (E8) edge (D7)
        ;
\end{tikzpicture}
\caption{Subordination diagram of caustics (taken from
    \cite[Figure 1]{Duistermaat:Oscillatory}). \label{fig:subordination}}
\end{figure}

\yellow{We will employ the subordination diagram to prove
  Theorem~\ref{theorem:sufficient} by proceeding inductively to the right
  through the columns of the diagram, proving at each stage that the theorem holds
  for a new set of singularity types based on its validity for all
  subordinate types.}

This means that the steps of our induction are:
\begin{enumerate}
\item $A_1$ (no singularity)
\item $A_2$
\item $A_3$
\item $A_4,$ $D^\pm_4$
\item $A_5,$ $D_5$
  \item $A_6,$ $D^\pm_6,$ $E_6$
  \item $A_7, D_7,E_7$
  \item $A_8,$ $D^\pm_8,$ $E_8.$
\end{enumerate}

\yellow{Another ingredient in our arguments will be the quasi-homogeneity of the phase functions.}
We reproduce for the reader's convenience a table from Duistermaat \cite[Table 4.3.2]{Duistermaat:Oscillatory}
showing the homogeneities of the parametrizing functions $f_j$ and $f$ from
Table~\ref{table:classification}.  These are the exponents $r_j$ and
$s_j$ such that
\begin{equation}\label{totalhomogeneity}
\phi(\lambda^{1-s_1} x_1,\dots, \lambda^{1-s_n} x_n,
\lambda^{r_1}\theta_1,\dots, \lambda^{r_k}\theta_k)=\lambda \phi(x,\theta).
\end{equation}

Note that these homogenities arise as follows in the parametrizations
given above: the $r_\ell$ are simply the inverses of the homogeneities
of the terms in $f(\theta)$ and then the $s_\ell$ are computed by
writing the monomials $f_\ell(\theta)$ as $\theta_1^{s_{\ell 1}}\dots
\theta_k^{s_{\ell k}}$ and then setting
$$
s_\ell=\sum_{j=1}^k s_{\ell j} r_j.
$$

\def\arraystretch{1.5}
\begin{table}
\begin{equation*}
\begin{array}{||l|l|l||}\hline\hline
  \text{Type}\ &  r_1,\dots r_k & s_1,\dots, s_k \\ \hline
A_{m+1} & \frac 1{m+2}, \frac 12, \dots, \frac 12 & \frac
                                                    1{m+2},\dots,
                                                    \frac{m}{m+2},
                                                    0,\dots, 0\\ \hline
  D^\pm_{m+1} & \frac 12 -\frac{1}{2m}, \frac 1m, \frac 12, \dots,
                \frac 12& \frac 12-\frac 1{2m},\frac 1m, \frac 2m,
                          \dots, \frac {m-1}m, 0, \dots, 0\\ \hline
  E_6 &  \frac 13, \frac 14, \frac 12, \dots, \frac 12& \frac 13,
                                                        \frac 14,
                                                        \frac 12,
                                                        \frac 7{12},
                                                        \frac 56, 0,
                                                        \dots, 0\\ 
  \hline E_7 & \frac 13,\frac 29, \frac 12, \dots \frac 12 & \frac 13, \frac 29, \frac 49, \frac 69, \frac 89, \frac 59,0,\dots 0 \\
  \hline E_8 & \frac 13,\frac 15, \frac 12, \dots \frac 12 & \frac 13, \frac 15, \frac 25, \frac 35, \frac{8}{15}, \frac{11}{15},\frac{14}{15},0,\dots 0 \\
\hline\hline
\end{array}
\end{equation*}
\caption{Homogeneities in the parametrizations of the stable singularities.}\label{table:homog}
\end{table}

These homogeneities lead directly to the orders of the relevant
caustics, which are given by
$$
\kappa=\frac 12 k-\sum_{j=1}^k r_j
$$
in each case---see Table~\ref{table:intro-orders} above.  We are able to prove that our $L^2$-normalized
Lagrangian distributions have $\sup$-norm bounds given by
$O(h^{-\kappa})$ in each case up to some threshold value of $\delta$ in
our symbol regularity estimates, and this threshold, interestingly,
seems to depend (at least in our proof) on not just the
homogeneities of the caustic in question from the table above, but
indeed on the homogeneities of caustics \yellow{subordinate to it in
Figure~\ref{fig:subordination}.} 
Our threshold $\delta_0$ for
a singularity (see Table~\ref{table:intro-orders}) is
the \emph{minimum} of the homogeneities $r_j$ for all caustics encountered as
\yellow{we move to the left along arrows of the subordination diagram.}
Note the distinction in thresholds between $D_{m+1}^\pm$ \yellow{for
  $m$ odd} occurs
\yellow{because $A_m$ is subordinate to $D_{m+1}^-$ but not to $D_{m+1}^+.$}
Meanwhile, the orders in our Table~\ref{table:intro-orders} simply match the orders of caustics in \cite[Table
4.3.2]{Duistermaat:Oscillatory}.

\yellow{With these preliminaries in hand, we proceed to the proof of the theorem.}
\begin{proof}
\yellow{We inductively show that if the result holds for all subordinate types
to the singularity parametrized by $\phi$ then it holds for the
singularity of $\phi$ as well. (The base case of the induction will be discussed at the
end.)}
  
\yellow{Following the notation employed above (and in the proof of \cite[Proposition 4.3.1]{Duistermaat:Oscillatory}), we let
  $k_0$ denote the number of nonzero $f_j$ in the parametrization of
  $\phi$ given in Table~\ref{table:classification}, which is equal to the
  codimension of the equisingularity manifold.  Hence $\phi$ is
  in fact independent of all \yellow{$x$} variables except $x_1,\dots x_{k_0}.$}

  \blue{For $a  \in (0, \infty),$ set}
\begin{equation}\label{Omegadef}
\Omega(a) = \{x: \sum_{j=1}^{k_0} \smallabs{x_j}^{\frac 1{1-s_j}} \leq a\}.
\end{equation}

For simplicity of notation we will use multi-index notation for the
scalings, so, e.g.\
$$
\mu^r \theta\equiv (\mu^{r_1} \theta_1,\dots, \mu^{r_k} \theta_k),
$$
and
$$
\smallabs{r} \equiv \sum_{j=1}^k r_j.
$$
\blue{We will also write
\begin{equation}\label{xscaling}
\mu^{1-s }x=(\mu^{1-s_1} x_1 ,\dots 
\mu^{1-s_{k_0}}x_{k_0}, x_{k_0+1}, \dots x_n),
\end{equation}
i.e., the scaling in this case only applies to the first $k_0$ coordinates.}

We first make a change of variables $\theta=h^r \eta$ (in the
multiindex notation just introduced).  By homogeneity \eqref{totalhomogeneity},
\begin{equation}\label{hom2}
f(\mu^r \eta) =\mu f(\eta)\red{,}\qquad f_j(\mu^r \eta) = \mu^{s_j} f_j(\eta),
\end{equation}
so that
$$
I(x) = h^{-k/2+\smallabs{r}}\int a(x,h^r \eta)\, e^{i \sum h^{s_j-1}
  x_j f_j(\eta)} e^{i f(\eta)}\, d\eta.
$$
Recall from \cite[p.263]{Duistermaat:Oscillatory} that $\nabla f(\eta)=0$ only at $\eta=0.$
Hence we may
obtain convergence of the integral by integrating
by parts repeatedly using the first order differential operator
$$
L_0 \equiv (1+\smallabs{\nabla f}^2)^{-1} (1+\sum \pa_j f(\eta) D_j)
$$
which has the property that $L_0 e^{i f(\eta)}=e^{i f(\eta)}.$
Application of this operator to the remaining parts of the phase
entails no loss in powers of $h$ 
 as long as
$\delta\leq r_j$ (our standing assumption), and
$$
\smallabs{x_j} \leq h^{1-s_j} \text{ for all }\blue{ j=1,\dots, k_0},
$$
i.e., as long as $x \in \Omega(h)$ (as defined in \eqref{Omegadef}).
Moreover, provided $x \in\Omega(h),$ each factor of $L_0$ hitting the exponential term
$e^{i \sum h^{s_k-1}   x_k f_k(\eta)}$ is in fact a sum of terms each bounded
by a multiple of one of the expressions
\begin{equation}\label{decaying}
\frac{(\pa_j f) (\pa_j f_k)}{(\pa_j f)^2}
\end{equation}
outside a large ball.  From \eqref{hom2} we easily compute the
homogeneities of derivatives:
\begin{equation}\label{derivhom}
(\pa_j f)(\mu^r \eta)=\mu^{1-r_j} (\pa_j f)(\eta),\quad (\pa_j f_k)(\mu^r \eta)=\mu^{s_k-r_j} (\pa_j f_k)(\eta),
\end{equation}
hence \eqref{decaying} has negative homogeneities.  We also note that
terms where $L_0$ falls on $a$ have increased decay for similar reasons. Thus, iteration of
the integration by parts renders the integral convergent in $\eta.$
Hence we obtain
$$
I(x) = O(h^{-k/2+\smallabs{r}}) \text{ for } x \in \Omega(h),
$$
which suffices to prove the desired estimate for such values of $x,$
since $\kappa=-k/2+\smallabs{r}.$

It thus remains to prove the desired estimate for $x \in
\Omega(R)\backslash \Omega(h)$ for some $R;$ without loss of
generality we may do a fixed rescaling to take $R=1.$  For any $x \in
\Omega(1) \backslash \Omega(h),$ there exists
$\lambda \in [h,1]$ such that if we set
$x_j=\lambda^{1-s_j} y_j$ \blue{(for $j=1,\dots, k_0$)} we now have $y \in \partial\Omega(1).$

Thus, employing the change of variables $\theta=\lambda^r \eta$ \blue{(in
the notation \eqref{xscaling}),} we obtain
$$
\begin{aligned}
I(\lambda^{1-s} y) &= h^{-k/2} \int a(\lambda^{1-s} y,\theta) e^{i\phi(\lambda^{1-s}
  y,\theta)/h} \, d\theta\\
 &= \lambda^{\smallabs{r}} h^{-k/2}\int a(\lambda^{1-s} y,\lambda^{r} \eta) e^{i\phi(\lambda^{1-s}
   y, \lambda^r \eta)/h} \, d\eta\\
&= \lambda^{\smallabs{r}} h^{-k/2} \int a(\lambda^{1-s} y,\lambda^{r} \eta) e^{i(\lambda/h)\phi( y, \eta)} \, d\eta.
\end{aligned}
$$

We split $I(\lambda^{1-s} y)$ into two pieces by letting $\chi\in
\mathcal{C}_c^\infty(\RR)$ equal $1$ on $(-1, 1)$ and $0$ on
$\RR\backslash (-2,2),$ picking $R>0,$ and expressing $$I(\lambda^{1-s} y)=J_<(\lambda^{1-s} y)+J_>(\lambda^{1-s} y)$$ with 
$$\begin{aligned}
J_<(\lambda^{1-s} y)&\equiv \lambda^\smallabs{r} h^{-k/2} \int \chi(\smallabs{\eta}/R)
a(\lambda^{1-s} y,\lambda^{r} \eta) e^{i(\lambda/h)\phi( y, \eta)} \,
d\eta,\\
J_>(\lambda^{1-s} y)&\equiv \lambda^\smallabs{r} h^{-k/2}\int(1- \chi(\smallabs{\eta}/R)) a(\lambda^{1-s} y,\lambda^{r} \eta) e^{i(\lambda/h)\phi( y, \eta)} \, d\eta,
\end{aligned}
$$

We once again remark that by \eqref{derivhom}, since $\nabla f(\eta)$ is
nonzero for $\eta \neq 0$ and has larger homogeneity than the $\nabla
f_j$'s (since all $s_j<1$), $\nabla_\eta \phi$ is nonzero on
sufficiently large quasi-homogeneous balls, hence if $R$ is
sufficiently large, the denominator of the operator
$$
L_1=\frac{h}{\lambda} \smallabs{\nabla_\eta \phi}^{-2} \sum (\pa_{\eta_j} \phi) D_{\eta_j},
$$
is nonvanishing, and we may use it to integrate by parts in our
expression for $J_>.$  Derivatives falling on the $a$ term each yield
a factor bounded \red{uniformly in $y\in\partial\Omega(1)$} by $(h/\lambda) \lambda^{r_j} h^{-\delta}$ for some
$j;$ since $\lambda<1$ and $\delta\leq r_j,$ this is bounded by
$(h/\lambda)^{1-\delta}$ \red{uniformly in $y\in \partial\Omega(1)$}.  Derivatives falling on the cutoff $\chi$ of course
yield $(h/\lambda),$ which is smaller yet since $h/\lambda\leq 1$.
If we employ a high enough power of $L_1,$ we moreover obtain
convergence of the integral in $\eta,$ again by considerations of homogeneity.
Hence for all $N \red{\geq N_0}$,
$$
J_>(\lambda^{1-s} y) = O( \lambda^{\smallabs{r}} h^{-k/2} (h/\lambda)^{N(1-\delta)})
$$
\red{where $N_0$ only depends on the phase function $\phi$.}
Recalling that $\kappa=k/2-\smallabs{r}$ and that $h/\lambda\leq 1$,
we thus may choose $N\red{\geq N_0}$ to obtain
$$
J_>(\lambda^{1-s} y)=O(h^{-\kappa}),
$$
\red{uniformly for $y\in \partial\Omega(1)$ and $\lambda\in [h,1]$ and hence uniformly for} $x \in \Omega(1)\backslash \Omega(h).$

We now turn to estimating $J_<(\lambda^{1-s}y).$  This term does have
a stationary phase.  To estimate it, we rewrite
$$\begin{aligned}
J_<(\lambda^{1-s} y)&= \lambda^\smallabs{r} h^{-k/2}\big(\frac{h}{\lambda}\big)^{+k/2}\big(\frac{h}{\lambda}\big)^{-k/2}  \int \chi(\smallabs{\eta}/R)
a(\lambda^{1-s} y,\lambda^{r} \eta) e^{i(\lambda/h)\phi( y, \eta)} \,
d\eta\\ &\equiv \lambda^\smallabs{r} h^{-k/2}\big(\frac{h}{\lambda}\big)^{+k/2} K(\lambda^{1-s} y).
\end{aligned}
$$
Note then that the integral expression for $K(\lambda^{1-s} y)$ is once again of the type
that our theorem applies to, but with $(h/\lambda)$ replacing $h$ as
the small parameter, and where we are interested in taking $y$ \yellow{in}
$\pa \Omega(1),$ hence away from the {set  $\mathcal{E}=\{x:x_1=x_2=\ldots=x_{k_0}=0\}$ }, where the phase is most singular.
In particular, since $\lambda <1,$ we do still have $a(\lambda^{1-s} y,\lambda^{r} \eta) \chi(\smallabs{\eta}/R)
\in S^0_\delta,$ compactly supported, uniformly for $\lambda \in
[h,1].$  With $y$ constrained to be near $\pa\Omega(1),$ \teal{and
  hence away from $\mathcal{E},$ the \yellow{projection of the equisingularity manifold through
  the origin},} 
we are guaranteed that the phase must parametrize a singularity
\teal{strictly} further down the
subordination diagram (Figure~\ref{fig:subordination}); \yellow{cf.\
  \cite[Proposition~4.3.1]{Duistermaat:Oscillatory}}.  \yellow{Thus the phase
function $\phi$ is equivalent, locally near any $(y_0,\eta_0)$ at
which it is stationary with $y_0 \in \pa \Omega(1),$ to some other
phase function $\tphi$ where $\tphi$ is one of the phase functions
from Table~\ref{table:classification} parametrizing a singularity
subordinate to the one we started with.  \yellow{Since, as noted at the
beginning of this section, we may change phase function to an
equivalent one by making a change of phase variables (and a coordinate
transformation in the base), we may use our
inductive hypothesis to estimate}
$$
K(\lambda^{1-s} y) = O \big((h/\lambda)^{-\kappa'}\big).
$$
Here $\kappa'\leq \kappa,$ since moving down the subordination diagram reduces the
order of the caustic.}

Thus, recalling that $\kappa=k/2-\smallabs{r},$ and using the facts
that $\lambda \geq h$ and $\kappa\geq \kappa',$ we \red{reassemble} our
estimates for $J_{\gtrless}$ to obtain
$$\begin{aligned}
\smallabs{I(\lambda^{1-s} y)}& \leq C  \lambda^{\smallabs{r}}
h^{-k/2}\big(\frac{h}{\lambda}\big)^{+k/2}\big( \frac h \lambda
\big)^{-\kappa'}+ C h^{-\kappa}\\
&= C \lambda^{-\kappa+\kappa'} h^{-\kappa'}+C h^{-\kappa}\\
&\leq C h^{-\kappa}.
\end{aligned}
$$
To complete the induction, it suffices to establish Theorem \ref{theorem:main} for a  $\delta$-Lagrangian distribution $u_h$ that is microsupported on a projectable subset of the Lagrangian $\lag$. That is, it remains to establish the case $A_1$. The claimed order of $u_h$ in this case is $\kappa=0$, with threshold $\delta=1$. Due to the breakdown of stationary phase asymptotics for $\delta \geq 1/2$, it is simplest to  use the particular representation $u_h=ae^{i\phi/h}$ obtained in Proposition \ref{HM2} and Proposition \ref{HM2.quasi} for $\delta$-Lagrangian distributions and quasimodes respectively. In either case, we have $\|u_h\|_{L^\infty}=\|a\|_{L^\infty}$, and the desired estimates on $\|u_h\|_{L^\infty}$ follow.
\end{proof}

\section{Beyond the $\delta_0$ threshold}\label{sec:beyond}

In this section, we determine the sharp $L^\infty$ estimates for the
situation described in Theorem~\ref{theorem:main}, but now with
$\delta\in [\delta_0,1]$ beyond the threshold of that theorem. We work
in the simplest nontrivial case, that of the fold caustic
$A_2$ in $\RR^1.$ \yellow{This caustic is famously associated
  with the asymptotics of the Airy
  function \[\mathrm{Ai}(x)=\frac 1{2\pi}\int_{-\infty}^\infty e^{i(x\theta+\theta^3/3)}\,
    d\theta,\] as the phase of the integral has an $A_2$ singularity.}

\begin{theorem}\label{theorem:beyond}
Let $\delta \in [0,1].$  Let $u_h$ be a compactly supported $\delta$-Lagrangian distribution with respect to the
Lagrangian $$\{x=\xi^2\} \subset T^*\RR^1 $$  Then there exists $C=C_\delta$ such that for $h \in (0,1)$, we have
\red{
\begin{equation}
\label{eq:beyond}
\frac{\smallnorm{u_h}_{L^\infty}}{\smallnorm{u_h}_{L^2}} \leq
\begin{cases}
$$Ch^{-(1+3\delta)/6}\quad \textrm{if }\delta\in [0,1/3]$$\\
$$Ch^{-(1+\delta)/4}\quad \textrm{if }\delta\in [1/3,1]$$
\end{cases}
\end{equation}
}
and \red{these} estimate\red{s are} sharp.
\end{theorem}
\begin{proof}
\red{Theorem~\ref{theorem:main} for the singularity $A_2$ in dimension $n=1$ gives the estimate 
	\begin{equation}
	\label{eq:below}\frac{\|u_h\|_{L^\infty}}{\|u_h\|_{L^2}}=O(h^{-1/6-\delta/2})
\end{equation} 
for $\delta\in[0,1/3]$, which is saturated at $x=0$ by the example
$$u_h(x)=\int_\RR \chi(\theta/h^\delta)e^{i(x\theta+\theta^3)/h}\, d\theta $$
where $\chi\in\mathcal{C}_c^\infty(\RR)$ is a $h$-independent bump function, nonvanishing at $0$. 
To see this, we observe that $\|u_h\|_{L^2}=O(h^{1/2+\delta/2})$ by Plancherel, and $u_h$ is a $\delta$-Lagrangian distribution by \eqref{dlagcriterion}. Direct computation then yields 
$$u_h(0)=\int_\RR \chi(\theta/h^\delta)e^{i\theta^3/h}\, d\theta= h^{1/3}\int_\RR \chi(h^{1/3-\delta}\theta)e^{i\theta^3}\, d\theta \sim Ch^{1/3}$$
for $C\neq 0$ as $h\downarrow 0$.
Hence \begin{equation}
\frac{\|u_h\|_{L^\infty}}{\|u_h\|_{L^2}}\gtrsim \frac{h^{1/3}}{h^{1/2+\delta/2}}=h^{-1/6-\delta/2}.
\end{equation}
At the threshold $\delta=1/3$, \eqref{eq:below} coincides with \eqref{eq:beyond}, and both give the bound $$\frac{\|u_h\|_{L^\infty}}{\|u_h\|_{L^2}}=O(h^{-1/3}).$$}

We now prove \red{the estimate \eqref{eq:beyond} in the case $\delta\geq 1/3$}. Recall from the beginning of the proof of Proposition~\ref{HM1} that
for \emph{any} $\delta \in [0,1]$ if we parametrize our Lagrangian
with the special phase function
$\teal{\phi(x,\theta)=x\theta-\theta^3/3},$
we arrive at the oscillatory integral representation
\begin{equation}\label{uform}u_h(x; h) =\int_\RR a(\theta;h)e^{i(x\theta-\theta^3/3)/h}\,
d\theta, \end{equation} where $a\in \mathcal{C}_c(\RR)$ satisfies the
estimate $\|\partial^\alpha a\|_{L^2}=O(h^{-1/2-\delta|\alpha|})$
\eqref{vreg} as
well as the Sobolev embedding estimate \eqref{vest}
$\|\partial^\alpha a\|_{L^\infty}=O(h^{-(1+\delta)/2})$. (Note that
in the notation of \eqref{vreg}, \eqref{vest}, we have $a=h^{-1}v_h,$
with the factor of $h^{-1}$ arising from the inverse semiclassical
Fourier transform.)

We \red{now integrate by parts in \eqref{uform} using an $h$-dependent regularization of the operator $(x-\theta^2)^{-1}hD_\theta$ which stabilizes the exponential factor in the integrand, but is singular at the caustic. To this end, we }introduce the differential operator
$$\teal{L=(x-\theta^2+ih^{1-\delta})^{\red{-1}}(hD_\theta +ih^{1-\delta})}.$$ This operator stabilizes the exponential factor in
the integrand and has transpose
$$\teal{L^T=(x-\theta^2+ih^{1-\delta})^{-1}(-hD_\theta+ih^{1-\delta})-\frac{2h\theta}{(x-\theta^2+ih^{1-\delta})^2}.}$$
Integration by parts shows $u_h(x)$ is bounded above by
\begin{equation*}
h\int_\RR \frac{|D_\theta a|}{|x-\theta^2+ih^{1-\delta}|}\, d\theta + h^{1-\delta}\int_\RR \frac{|a|}{|x-\theta^2+ih^{1-\delta}|}\, d\theta +2h\int_\RR \frac{|\theta a|}{|x-\theta^2+ih^{1-\delta}|^2}\, d\theta\label{ibpest}
\end{equation*}
Using Cauchy--Schwarz, we obtain
\begin{align*}
\smallabs{u_h(x)} &\lesssim (h\|D_\theta a\|_{L^2} +h^{1-\delta} \|a\|_{L^2})\left(\int_\RR \frac{1}{(x-\theta^2)^2+h^{2-2\delta}}\, d\theta\right)^{1/2}\\
&+ h\|a\|_{L^\infty} \int_\RR \frac{|\theta|}{(x-\theta^2)^2+h^{2-2\delta}} \, d\theta  \\
&\lesssim h^{1/2-\delta} \left(\int_\RR
                                                                                             \frac{1}{(x-\theta^2)^2+h^{2-2\delta}}\,d\theta\right)^{1/2} + h^{(1-\delta)/2}\int_\RR \frac{|\theta|}{(x-\theta^2)^2+h^{2-2\delta}}\, d\theta.
\end{align*}
We now estimate these integrals as follows.
\noqed
\end{proof}
\begin{lemma}
	\label{lemma:intest}
	We have the following two integral estimates, uniform for $x
        \in \RR$
         as $\epsilon\rightarrow 0^+$.
	\begin{align}
	\label{eq:intest1}\int_\mathbb{R} \frac{1}{(x-\theta^2)^2+\ep^2}\, d\theta&= O(\ep^{-3/2})\\
	\label{eq:intest2}\int_\mathbb{R} \frac{|\theta|}{(x-\theta^2)^2+\ep^2} \, d\theta&= O(\ep^{-1}).
	\end{align}
      \end{lemma}
\begin{proof}
We evaluate the first integral by changing variables to set
$\eta=\theta \ep^{-1/2}.$  This yields
$$
\teal{\ep^{-3/2} M(-x\ep^{-1})},
$$
where
$$
M(\alpha)\equiv \int_{-\infty}^\infty \frac{1}{(\eta^2+\alpha)^2+1}\, d\eta.
$$
It thus suffices to show that $\sup_{\alpha \in
  \RR}\smallabs{M(\alpha)}<\infty.$  Indeed $M$ is manifestly
uniformly bounded
for $\alpha\geq 0;$ to deal with negative $\alpha,$ we note that the
integral can be evaluated explicitly by contour integration to yield
$\pi \Re (\alpha+i)^{-1/2},$ which is indeed uniformly bounded for
$\alpha \in \RR.$ 

The integral \eqref{eq:intest2} is simply
\teal{	\begin{align*}
	\int_\mathbb{R} \frac{|\theta|}{(x-\theta^2)^2+\ep^2}\,d\theta &= 2\int_0^\infty\frac{\theta}{(x-\theta^2)^2+\ep^2}\,d\theta \\
	&= \ep^{-1}\left[\arctan\left(\frac{\theta^2-x}{\ep}\right)\right]^{\infty}_0\\
	&\leq \pi \ep^{-1}\qed
	\end{align*}
}
        \noqed
  \end{proof}

As a consequence of these estimates, and since we are taking $\delta
\geq 1/3,$  we now obtain
\begin{align*}
\smallabs{ u_h(x)} & \lesssim h^{1/2-\delta}\cdot h^{-3/4+3\delta/4}+h^{(1-\delta)/2}\cdot h^{\delta-1}\\
&= O(h^{-(1+\delta)/4})
\end{align*}
uniformly for $x \in\RR$  for any $\delta \geq 1/3.$  This is the
desired upper bound.

To show that the estimate is sharp, we simply remark that our estimate
is saturated by the $\delta$-Lagrangian distribution given by \eqref{uform} with amplitude 
$$a=h^{(\delta-3)/4}\chi(\theta/h^{(1-\delta)/2})e^{i\theta^3/3h}$$
where $\chi$ is a $h$-independent bump function, nonvanishing at $0$.
This $a$ has $L^2$ norm $O(h^{-1/2}),$ hence $\smallnorm{u_h}_{L^2}$ is
uniformly bounded, by Plancherel.  Moreover, \teal{we have 
\begin{equation}
\|\partial_\theta^\alpha a\|_{L^2}=O(h^{-1/2-\delta|\alpha|})
\end{equation}} as
$\theta^2/h \leq h^{-\delta}$ in the support of $a,$ hence \teal{$u_h$}
is indeed \teal{an $L^2$ bounded} $\delta$-Lagrangian distribution \teal{by \eqref{dlagcriterion}}.
On the other hand, we may explicitly compute
$$u(0)=h^{(\delta-3)/4}\int \chi(\theta/h^{(1-\delta)/2})\gtrsim
h^{(\delta-3)/4}\cdot h^{(1-\delta)/2}=h^{-(1+\delta)/4},$$ thereby
saturating our upper bound.\qed

\bibliographystyle{abbrv} 
\bibliography{all}

\end{document}